\documentclass[12pt,leqno]{article}
\usepackage{amsmath,amssymb,amsthm,amscd,dsfont}
\numberwithin{equation}{section} \allowdisplaybreaks
\begin{document}
\newtheorem{theorem}{Theorem}[section]
\newtheorem{defin}{Definition}[section]
\newtheorem{prop}{Proposition}[section]
\newtheorem{corol}{Corollary}[section]
\newtheorem{lemma}{Lemma}[section]
\newtheorem{rem}{Remark}[section]
\newtheorem{example}{Example}[section]
\title{Towards a double field theory on para-Hermitian manifolds}
\author{{\small by}\vspace{2mm}\\Izu Vaisman}
\date{}
\maketitle
{\def\thefootnote{*}\footnotetext[1]%
{{\it 2000 Mathematics Subject Classification: 53C15, 53C80} .
}}
\begin{center} \begin{minipage}{12cm}
A{\footnotesize BSTRACT. In a previous paper, we have shown that the geometry of double field theory has a natural interpretation on flat para-K\"ahler manifolds. In this paper, we show that the same geometric constructions can be made on any para-Hermitian manifold. The field is interpreted as a compatible (pseudo-)Riemannian metric. The tangent bundle of the manifold has a natural, metric-compatible bracket that extends the C-bracket of double field theory. In the para-K\"ahler case this bracket is equal to the sum of the Courant brackets of the two Lagrangian foliations of the manifold. Then, we define a canonical connection and an action of the field that correspond to similar objects of double field theory. Another section is devoted to the Marsden-Weinstein reduction in double field theory on para-Hermitian manifolds. Finally, we give examples of fields on some well-known para-Hermitian manifolds.} \end{minipage} \end{center}
\vspace*{5mm}
\section{Introduction}
Since several years double field theory has been a subject of much interest in the research in theoretical physics that involves T-duality. We refer the reader to \cite{{HHZ},{HLZ}} for an account of the subject (see also the much earlier paper \cite{Duff}). In \cite{V0}, we have explained that the geometric framework of double field theory may be identified with a flat para-K\"ahler manifold. The aim of the present paper is to show that an action similar to that of double field theory can be constructed on any para-Hermitian manifold. We do not intend to discuss whether there is any physical interest in the obtained action functional and this question remains to be answered by physicists. In \cite{HZ13}, the authors propose an invariant framework for double field theory where double space is characterized by ``generalized coordinate transformations" and define a non-canonical connection, with the same torsion condition that we will use, and a different curvature tensor.

In the present paper everything belongs to the $C^\infty$ category and we follow the usual notation of differential geometry, e.g., \cite{KN}.

For a survey on para-complex and para-Hermitian geometry and for the literature on this subject, we refer to \cite{AMT}.
An almost para-complex structure on a differentiable manifold $M^{2n}$ is a tensor field $F\in End(TM)$ such that $F^2=Id$ and the $\pm1$-eigenbundles of $F$ have the equal rank $n$. Hence, the manifold must have an even dimension $2n$. A metric tensor is a $2$-covariant, symmetric, non degenerate tensor field $\gamma$, and $\gamma$ is compatible with the almost para-complex structure $F$ if
\begin{equation}\label{Fskew} \gamma(FX,Y)=-\gamma(X,FY).\end{equation}
If (\ref{Fskew}) holds, $(F,\gamma)$ is an almost para-Hermitian structure, the $2$-form
$$ \omega(X,Y)=\gamma(X,FY)$$
is the fundamental form, and the structure is almost para-K\"ahler if $d\omega=0$.

We will denote by $L=L_+,\bar{L}=L_-$ the $\pm1$-eigenbundles of $F$; the double notation is introduced for more graphic convenience. By (\ref{Fskew}), $L_\pm$ are $\gamma$-isotropic, hence $\gamma$ must be of signature zero ($\gamma$ is a {\it neutral metric}). If $L_\pm$ are integrable, i.e., tangent to foliations $\mathcal{L}=\mathcal{L}_+,\bar{\mathcal{L}}=\mathcal{L}_-$ (sometimes $L_\pm$ themselves are called foliations), the word ``almost" will be deleted from the name of the structure.

The subbundles $L_\pm$ are isotropic for the form $\omega$ as well. Thus, in the para-K\"ahler case, the foliations $\mathcal{L}_\pm$ are Lagrangian foliations, which is why para-K\"ahler manifolds are also called bi-Lagrangian manifolds \cite{AMT}. We extend the use of the term ``Lagrangian" from symplectic to almost symplectic manifolds (where $d\omega$ does not need to be zero), and say that $L_\pm$ are Lagrangian subbundles (foliations) in the para-Hermitian case as well, but we do not extend the use of the term ``bi-Lagrangian manifold".

Obviously, if a differentiable manifold has an almost para-Hermitian structure, there exists a corresponding reduction of the structure group of the tangent bundle to $O(n,n)$. This situation is analogous to that encountered in generalized geometry (e.g., \cite{G}), which suggests that it should be interesting to study Riemannian metrics that further reduce the structure group to $O(n)\times O(n)$; we call them compatible Riemannian metrics. In the literature on double field theory it was shown that a given field, consisting of a Riemannian metric and a $2$-form on space-time, is equivalent with a compatible metric on the double manifold, which has a natural flat para-K\"ahler structure. The case of fields with a pseudo-Riemannian metric may occur but, compatibility has a different definition to be given later on.

Accordingly, if a general para-Hermitian manifold is going to replace the double of space-time, a field will be defined as a compatible metric and we will recover the objects playing the role of the metric and form components of the field in Section 2. In Section 3, we extend the notion of a C-bracket \cite{HHZ}. The extended bracket is the $\gamma$-compatible bracket defined by the Levi-Civita connection and we will show that, in the para-K\"ahler case, it is equal to the sum of the Courant brackets of the Lagrangian foliations $\mathcal{L}_\pm$. In Section 4, we construct a canonical connection that preserves the metric $\gamma$ and has a vanishing torsion-type invariant defined with the extended C-bracket. There exists a scalar curvature-type invariant and the integral of the latter produces an action of the field. In Section 5, we discuss Marsden-Weinstein reduction by a group of symmetries with an equivariant momentum map in double field theory on a para-Hermitian manifold. In the last Section 6, we give examples of compatible metrics on the most known para-Hermitian manifolds: $G\times G$ (where $G$ is a Lie group), the cotangent bundle of a flat, pseudo-Riemannian manifold and the para-complex projective models \cite{GA}.
\section{Fields and field components}
We begin with the following definition (the word``pseudo" between parentheses is used to indicate that both cases may occur)
\begin{defin}\label{def1} {\rm A {\it field}, equivalently, a {\it compatible (pseudo-)Riemannian metric} on the para-Hermitian manifold $(M^{2n},\gamma,F)$ is a (pseudo-)Riemannian (hence, non degenerate) metric	$g$ such that $g|_{L_-}$ is non degenerate and
\begin{equation}\label{compat} \sharp_g\circ\flat_\gamma= \sharp_\gamma\circ\flat_g,\end{equation}
where the musical isomorphisms are those defined in Riemannian geometry.} \end{defin}

The restriction $g|_{L_-}$ is automatically non degenerate in the Riemannian case. The choice of $L_-$ rather than $L_+$ is made in order to facilitate comparison with the literature on double field theory. The subbundles $L_\pm$ play a similar role, but the choice of the tensor fields $F$ distinguishes between them and to exchange $L_+$ with $L_-$ we have to change $F$ into $-F$. A manifold with a structure $(F,\gamma,g)$ as in Definition \ref{def1} will also be called a {\it(pseudo-)Riemannian, para-Hermitian manifold} ({\it para-K\"ahler}, if $d\omega=0$). Finally, notice that Definition \ref{def1} may be formulated on almost para-Hermitian manifolds but we are not interested in this case.
\begin{prop}\label{propH} A compatible metric $g$ is equivalent with an almost product structure $H$ on $M$ such that the tensor field
\begin{equation}\label{gdinH} g(X,Y)=\gamma(HX,Y)\end{equation}
is symmetric, non degenerate, and has a non degenerate restriction to $L_-$. Furthermore, $H$ is an almost para-complex structure, and $g$ provides a further reduction of the structure group of $TM$ to a group of the form $O(p,q)\times O(q,p)$, where $p+q=n$.
\end{prop}
\begin{proof} Notice that in the Riemannian case we just have to require that $g$ be symmetric and positive definite. Generally,
given $g$, we define $H=\sharp_g\circ\flat_\gamma\in End(TM)$. Then, (\ref{compat}) is equivalent with $H^2=Id$, hence, $H$ is an almost product structure and we will denote by $V_\pm$ the $\pm1$-eigenbundles of $H$. The definition of $H$ and (\ref{compat}) imply (\ref{gdinH}). Furthermore, we can see that the condition that $g|_{L_-}$ be non degenerate is equivalent with $L_-\cap V_\pm=0$. Indeed, since $L_-$ is $\gamma$-isotropic, we have $$g(v,w)=\gamma(Hv,w)=\pm\gamma(v,w)=0,\;\; v\in V_\pm\cap L_-,w\in L_-.$$
Thus, $g|_{L_-}$ non degenerate implies $L_-\cap V_\pm=0$. For the converse, assume that $L_-\cap V_\pm=0$. If $v\in L_-$ is $g$-orthogonal to $L_-$, then, $Hv$ is $\gamma$-orthogonal to $L_-$, therefore, $Hv\in L_-$ and $pr_{V_\pm}v=(1/2)(v\pm Hv)\in L_-\cap V_\pm$. Thus, $Hv=0$ and $v=0$, which shows that $g|_{L_-}$ is non degenerate. (Similarly, $g|_{L_+}$ non degenerate is equivalent with $L_+\cap V_\pm=0$.) Since $rank\,L_-=n$, $L_-\cap V_\pm=0$ implies $rank\,V_\pm\leq n$ and because $TM=V_+\oplus V_-$, we must have $rank\,V_\pm= n$. Therefore, $H$ is an almost para-complex structure. Moreover, if $g|_{V_\pm}$ were degenerate, the decomposition $TM=L_-\oplus V_\pm$ would imply the degeneracy of $g$ on $TM$, which is contrary to the hypotheses. Now, notice that (\ref{gdinH}) also implies
\begin{equation}\label{ggamma} g|_{V_\pm}=\pm\gamma|_{V_\pm},\hspace{2mm} V_+\perp_g V_-,\hspace{2mm}V_+\perp_\gamma V_-.\end{equation} Accordingly, $\gamma|_{V_\pm}$ are non degenerate and, if $p,q$, $p+q=n$, are the positive-negative inertia indices of $\gamma|_{V_+}$, the corresponding indices of $\gamma|_{V_-}$ are $q,p$. This remark and (\ref{ggamma}) justify the required reduced structure group.

Finally, it is obvious that if $H$ satisfies the hypotheses of the proposition, formula (\ref{gdinH}) provides a metric that satisfies the conditions of Definition \ref{def1}.

Notice also the compatibility relations
\begin{equation}\label{ggammaH} g(HX,HY)=g(X,Y),\;\gamma(HX,HY)=\gamma(X,Y).
\end{equation}
\end{proof}

Now, we will prove that a field in the sense of Definition \ref{def1} may be interpreted as a pair that consists of a metric and a $2$-form along the leaves of the foliation $L$. The proof follows a known procedure of generalized geometry \cite{{G},{V0}}.
\begin{prop}\label{existence} On any (almost) para-Hermitian manifold, there exists a bijective correspondence between the compatible, (pseudo-)Riemannian metrics $g$ and the pairs $(k,\beta)$ where $k$ is a (pseudo-)Euclidean metric on $L$ and $\beta$ is a $2$-form on $L$.
\end{prop}
\begin{proof}
The endomorphism $H$ has a matrix representation
\begin{equation}\label{matriceaH} H\left(\begin{array}{c}
X_+\vspace{2mm}\\ X_-\end{array}\right)=
\left(\begin{array}{cc}\psi_+&\tilde{\theta}\vspace*{2mm}\\
\theta&\psi_-\end{array}\right)
\left(\begin{array}{c} X_+\vspace{2mm}\\ X_-\end{array}\right),
\end{equation} where $X_\pm\in L_\pm$, $\psi_\pm\in End(L_\pm)$, $\theta\in Hom(L_+,L_-)$, $\tilde{\theta}\in Hom(L_-,L_+)$.
Correspondingly, the metric $g$ may be written under the form
\begin{equation}\label{gcumatricea} \begin{array}{l}
g(X_+,Y_+)=\gamma(\theta X_+,Y_+),\,
g(X_-,Y_-)=\gamma(\tilde{\theta} X_-,Y_-),\vspace*{2mm}\\
g(X_+,Y_-)=\gamma(\psi_+X_+,Y_-),\,
g(X_-,Y_+)=\gamma(\psi_-X_-,Y_+).\end{array}\end{equation}

The symmetry of the tensor $g$ defined by(\ref{gdinH}) means that $\theta,\tilde{\theta}$ produce symmetric tensor fields $h_\pm=g|_{L_\pm}$ and $\psi_\pm$ are the transposed of $\psi_\mp$ with respect to the metric $\gamma$. The non degeneracy of $g|_{L_-}$ means that $\tilde{\theta}$ is an isomorphism. Furthermore, the product condition $H^2=Id$ is equivalent to
\begin{equation}\label{Hprodus} \begin{array}{l}
\psi_+^2+\tilde{\theta}\circ\theta=Id_{L_+},\,
\psi_-^2+\theta\tilde{\theta}=Id_{L_-},
\vspace*{2mm}\\ \psi_+\circ\tilde{\theta}+\tilde{\theta}\circ\psi_-=0,\,
\psi_-\circ\theta+\theta\circ\psi_+=0,
\end{array}\end{equation} where the second and fourth condition are implied by the first and third condition by $\gamma$-transposition.

Accordingly, we get a bijective correspondence between the required structures $H$ and the pairs $(\psi_-,\tilde{\theta})$ where $\tilde{\theta}$ is an isomorphism, which give $$\psi_+=-\tilde{\theta}\psi_-\tilde{\theta}^{-1},
\theta=(Id-\psi_-^2)\tilde{\theta}^{-1}.$$ The third condition (\ref{Hprodus}) is equivalent with the fact that
$$ \beta_-(X_-,Y_-)=\gamma(\tilde{\theta}\psi_-X_-,Y_-)$$
is a $2$-form on $L=L_-$. Thus, also looking at (\ref{gcumatricea}), the pair $(\psi_-,\tilde{\theta})$ may be replaced by the pair
$(g_-=g|_{L_-},\beta_-)$ where $g_-$ is a non degenerate metric on $L_-$ and $\beta_-$ is a $2$-form on $L_-$. This configuration is equivalent with a pair $(k,\beta)$ as required by putting
\begin{equation}\label{defk} \begin{array}{l}
k(X_+,Y_+)=2g_-(\tilde{\theta}^{-1}X_+
,\tilde{\theta}^{-1}Y_+)=2\gamma(\tilde{\theta}^{-1}X,Y),\vspace*{2mm}\\ \beta(X_+,Y_+)=2\beta_-(\tilde{\theta}^{-1}X_+
,\tilde{\theta}^{-1}Y_+)=2\gamma(\psi_-\theta^{-1}X_+,Y_+).\end{array}
\end{equation}
The last equality in (\ref{defk}) is a consequence of (\ref{gcumatricea}).
\end{proof}
\begin{defin}\label{def2} {\rm If $g$ is a field on the para-Hermitian manifold $(M,F,\gamma)$, the corresponding tensors $k,\beta$ are the {\it $L$-components} of the field.}\end{defin}

In string theory, the components $k,\beta$ correspond to the graviton and the Kalb-Ramond charge, respectively.
\begin{rem}\label{obsK} {\rm
From the proof of Proposition \ref{existence} we also see that a (pseudo-)Riemannian, almost para-Hermitian structure may be seen as a triple $(\gamma,F,H)$ where $(\gamma,F)$ is an almost para-Hermitian structure and $H$ is an almost product structure that satisfies the second condition (\ref{ggammaH}). In the Riemannian case one has to require that $g$ be positive definite. One more possibility is to see such a structure as a triple $(g,H,F)$ where $g$ is a metric, $H,F$ are as above and the following conditions hold
$$ g(HX,HY)=g(X,Y),\;g(KX,Y)=-g(X,KY),\hspace{2mm}K=HF.$$}\end{rem}

Now, since $L_-\cap V_\pm=0$, the projections $\tau_\pm=pr_{L}:V_\pm\rightarrow L=L_+$ defined by the decomposition $TM=L_+\oplus L_-$ are injections (if $v\in V_\pm$ and $\tau_\pm v=0$ then $v\in L_-$ too and $v=0$), therefore isomorphisms, with inverse mappings $\iota_\pm=\tau_\pm^{-1}:L\rightarrow V_\pm$, which allow a transfer of structures between the bundle $L$ and the bundles $V_\pm$ \cite{G}.
\begin{lemma}\label{lemaiota} A vector $X=X_++X_-$ belongs to $V_\pm$ iff
\begin{equation}\label{eqlema} X_-=-\tilde{\theta}^{-1}(\psi_+\mp Id)X_+.
\end{equation}\end{lemma}
\begin{proof} For $X\in V_\pm$ and any $Y_-\in L_-$, we have
$$\begin{array}{l} \gamma(X_+,Y_-)=\gamma(X,Y_-)=g(HX,Y_-)
\vspace*{2mm}\\ =\pm g(X,Y_-)=\pm g(X_+,Y_-)\pm g(X_-,Y_-)
\vspace*{2mm}\\ =\pm\gamma(\psi_+X_+,Y_-)\pm\gamma(\tilde{\theta} X_-,Y_-),
\end{array}$$
which implies $$ X_+=\pm\psi_+X_+
\pm\tilde{\theta} X_-.$$ Since $\tilde{\theta}$ is an isomorphism, the previous result is equivalent to (\ref{eqlema}).
Conversely, if (\ref{eqlema}) holds, going backwards in the previous calculation leads to
$$\gamma(X_+,Y_-)=\pm g(X,Y_-)=\pm\gamma(HX,Y_-)= \pm\gamma(pr_{L_+}HX,Y_-),$$
whence $pr_{L_+}(X\mp HX)=0$. Therefore, $X\mp HX=2pr_{V_\mp}X\in L_-\cap V_\pm$ must vanish and $X\in V_\pm$. \end{proof}
\begin{prop}\label{expriota} The isomorphisms $\iota_\pm:L\rightarrow V_\pm$ are given by
\begin{equation}\label{eqiota} \iota_\pm Z=Z-\tilde{\theta}^{-1}(\psi_+\mp Id)Z
\hspace{2mm}(Z\in L=L_+).\end{equation}
The pullback of the metrics $g|_{V_\pm}$ by $\iota_\pm$ to $L$ are equal to the metric $k$ of $L$. \end{prop}
\begin{proof} By definition, we have $(\iota_\pm Z)_+=Z$ and, since $(\iota_\pm Z)_-$ is given by (\ref{eqlema}), we get (\ref{eqiota}).
The second assertion is proven by the following calculation that uses
(\ref{eqiota}), the $\gamma$-symmetry of $\tilde{\theta},\tilde{\theta}^{-1}$, the $\gamma$-transposition between $\psi_+,\psi_-$ and conditions (\ref{Hprodus}):
$$ \begin{array}{l}
g(Z_1-\tilde{\theta}^{-1}(\psi_+\mp Id)Z_1,Z_2-\tilde{\theta}^{-1}(\psi_+\mp Id)Z_2)
\vspace*{2mm}\\ = \pm\gamma(Z_1-\tilde{\theta}^{-1}(\psi_+\mp Id)Z_1,Z_2-\tilde{\theta}^{-1}(\psi_+\mp Id)Z_2)
\vspace*{2mm}\\ =\mp(\gamma(Z_1,\tilde{\theta}^{-1}(\psi_+\mp Id) Z_2)+\gamma(Z_2,\tilde{\theta}^{-1}(\psi_+\mp Id)Z_1))
\vspace*{2mm}\\ =\mp(\gamma(Z_2,(\psi_-\mp Id)\tilde{\theta}^{-1}Z_1)+\gamma(Z_2,\tilde{\theta}^{-1}(\psi_+\mp Id)Z_1))
\vspace*{2mm}\\ =2\gamma(\tilde{\theta}^{-1}Z_1,Z_2)=k(Z_1,Z_2).
\end{array}$$
\end{proof}
\begin{corol} The compatible metric $g$ is Riemannian iff the metric $k$ is positive definite on $L$. \end{corol}
\begin{proof} The result follows from the proposition because the subbundles $V_\pm$ are $g$-orthogonal.\end{proof}

A particular type of compatible metrics is introduced by the following definition.
\begin{defin}\label{strictcomp} {\rm A compatible metric $g$ of the (almost) para-Hermitian manifold $(M,F,\gamma)$ is {\it strongly compatible} if the subundles $L=L_+,\bar{L}=L_-$ are orthogonal with respect to $g$.}\end{defin}

Since the intersection of the subbundles $L_\pm$ is $0$, if $g$ is strongly compatible, $g|_{L_\pm}$ necessarily are non degenerate. Furthermore, the $\gamma$-isotropy of $L_\pm$ implies $H(L_\pm)=L_\mp$ and, in (\ref{matriceaH}), we have $\psi_\pm=0,\tilde{\theta}=\theta^{-1}$. By the isotropy argument, we also see that a compatible metric such that $H(L)=\bar{L}$ is strongly compatible. Of course, each of the conditions $\psi_+=0$, $\psi_-=0$, equivalently $\beta=0$, is equivalent with strong compatibility; these conditions show that the metric $g$ bijectively corresponds to a non degenerate metric $k$ on $L$. The condition $H(L)=\bar{L}$ is equivalent with $FHF=-H$, i.e., $FH=-HF$, therefore, the latter is just another form of the strong compatibility condition. Furthermore, the anti-commutation of $H,F$ is equivalent to $K^2=-Id$ ($K=HF$) and $(g,K)$ is an almost Hermitian structure, while the pairs $(F,K),(H,K)$ are almost para-hypercomplex structures \cite{AMT}.

Examples of compatible metrics on some well known para-Hermitian manifolds will be given in the last section.
\section{The C-bracket}
In double field theory, in order to encompass T-duality, one uses a bracket of two vector fields that is different from the usual Lie bracket and is called the C-bracket \cite{{HHZ},{V0}}. In this section we show that this bracket extends to all para-Hermitian manifolds.

First, we recall some basic facts concerning Courant and metric algebroids \cite{{LWX},{U},{V0}}. Below, the upper index $*$ denotes the dual bundle, $\Gamma$ denotes the space of sections, $\odot$ denotes symmetric tensor product and the front index $t$ denotes transposition.
\begin{defin}\label{defalgL} {\rm Let $E\rightarrow M$ be a vector bundle  endowed with a non degenerate metric
$g\in\Gamma\odot^2E^*$ and an {\it anchor} morphism $\rho:E\rightarrow TM$, which produces the morphism $\partial=(1/2)\sharp_g\circ
\hspace*{1pt}^t\hspace*{-1pt}\rho:T^*M\rightarrow E$.
A {\it Dorfman bracket} is an $\mathds{R}$-bilinear operation $\bigstar:\Gamma E\times\Gamma E\rightarrow\Gamma E$ that satisfies the properties:
$$\begin{array}{l}
1)\;(\rho e)(g(e_1,e_2))=g(e\bigstar e_1,e_2)+g(e_1,e\bigstar e_2),\vspace*{2mm}\\
2)\;e\bigstar e=\partial(g(e,e))\hspace{2mm}(\partial f=\partial(df),\,\forall f\in C^\infty(M))\vspace*{2mm}\\ 3)\; e\bigstar(e_1\bigstar e_2)=(e\bigstar e_1)\bigstar e_2 +e_1\bigstar(e\bigstar e_2).
\end{array}$$
The quadruple $(E,g,\rho,\bigstar)$ is called a {\it Courant algebroid}. Axiom 1) is the $g$-{\it compatibility} axiom, axiom 2) is the {\it normalization} axiom and axiom 3) is the {\it Leibniz axiom}. If only axioms 1), 2) are required, the product is a {\it metric product} and the quadruple $(E,g,\rho,\bigstar)$ is a {\it metric algebroid}.}\end{defin}

The definition of $\partial$ is equivalent to
$$g(\partial f,e)=\frac{1}{2}(\rho e)f.$$
The metric product satisfies the properties \cite{{U},{V0}}
\begin{equation}\label{propcrg}\begin{array}{l}
a)\;e_1\bigstar(fe_2)=f(e_1\bigstar e_2)+((\rho e_1)f)e_2,\vspace*{2mm}\\
b)\;(fe_1)\bigstar e_2=f(e_1\bigstar e_2)-((\rho e_2)f)e_1 +2g(e_1,e_2)\partial f.\end{array}\end{equation}

Furthermore \cite{V0}, $(E,g,\rho)$ is a metric algebroid iff there exists an $\mathds{R}$-bilinear, skew symmetric bracket
$[\,,\,]:\Gamma E\times\Gamma E\rightarrow\Gamma E$, called a {\it metric bracket}, which satisfies the axiom
$$(\rho e)(g(e_1,e_2))=g([e,e_1]+\partial(g(e,e_1)),e_2)+g(e_1,[e,e_2]
+\partial(g(e,e_2))).$$
The product and the bracket reciprocally define each other by the relation
$$ [e_1,e_2]=e_1\bigstar e_2-\partial(g(e_1,e_2))$$
and properties a), b) are equivalent with
\begin{equation}\label{lincroset} [e_1,fe_2]=f[e_1,e_2]+(\rho e_1)(f)e_2
-g(e_1,e_2)\partial f.\end{equation}
If the product that corresponds to a given metric bracket also satisfies the Leibniz axiom, the bracket satisfies the identity
\begin{equation}\label{JacobiC} \sum_{Cycl}[[e_1,e_2],e_3]= \frac{1}{3}\partial\sum_{Cycl}g([e_1,e_2],e_3)\end{equation}
and it is a {\it Courant bracket}, while the algebroid is a Courant algebroid.

We also need to recall the Courant algebroid structure of the {\it big tangent bundle of a foliation} $ \mathcal{L}$ of a manifold $M$ \cite{VBuc}. This bundle is $\mathbf{L}=\mathbf{T}M^{\mathcal{L}}=L\oplus (T^*M/ann\,L)$ where $L=T\mathcal{L}$ and $ann$ denotes the annihilator bundle. The restriction of the Courant bracket of the big tangent bundle of the manifold $M^{\mathcal{L}}=\sum({\rm leaves\, of}\, \mathcal{L})$, defined in \cite{C}, to vector fields and $1$-forms that are differentiable with respect to the original $C^\infty$-structure of $M$ makes $\mathbf{L}$ into a Courant algebroid with anchor $pr_L$ and with the structure given by
$$ g_L((Y_1,\tilde{\alpha}_{1}),
(Y_2,\tilde{\alpha}_{2}))=\frac{1}{2}(\alpha_1(Y_2)+\alpha_2(Y_1)),
$$
$$ \partial f=\frac{1}{2}\sharp_{g_L}\hspace{1pt}^t\hspace{-1pt}\rho(df)=
(0,[df]_{ann\,L}),$$
$$ \begin{array}{rcl}
[(Y_1,\tilde{\alpha}_{1}),(Y_2,\tilde{\alpha}_{2})]&
=&([Y_1,Y_2],[(\mathfrak{L}_{Y_1}\alpha_2-
\mathfrak{L}_{Y_2}\alpha_1\vspace*{2mm}\\&+&\frac{1}{2}d(\alpha_1(Y_2)-
\alpha_2(Y_1))]_{ann\,L}),\end{array}$$
where $Y_1,Y_2\in\Gamma L$, $\alpha_1,\alpha_2\in\Omega^1(M)$, $f\in C^\infty(M)$, $\tilde{\alpha}=[\alpha]_{ann\,L}$ is the class modulo $ann\,L$ and $\mathfrak{L}$ is the Lie derivative; the results remain unchanged if $\alpha_{l}\mapsto
\alpha_{l}+\gamma_{l}$ with $\gamma_{l}\in ann\,L$.

Now, we refer to para-Hermitian manifolds and prove the following proposition.
\begin{prop}\label{propTMCourant} Let $(M,\gamma,F)$ be a para-Hermitian manifold. Then, the tangent bundle $TM$ is endowed with two natural structures of a Courant algebroid defined by its Lagrangian foliations $L,\bar{L}$.\end{prop}
\begin{proof}
Since the fundamental form $\omega$ defines a non degenerate pairing
$L_x\times\bar{L}_x\rightarrow\mathds{R}$ ($x\in M$),
we have an isomorphism $\flat_\omega:\bar{L}\rightarrow ann\,\bar{L}\approx L^*$. (As usually, $\flat_\omega X=i(X)\omega$ and we will also use $\sharp_\omega=\flat_\omega^{-1}$ below.)

Accordingly, the tangent bundle $TM=L\oplus\bar{L}$ may be identified with $\mathbf{L}$ by
$$ \begin{array}{l}
X\,\mapsto\,(X_L,\flat_{\omega} X_{\bar{L}})),\,(\flat_{\omega} X_{\bar{L}}\in ann\,\bar{L}),\vspace*{2mm}\\
(Y,\alpha)\,\mapsto\,Y+\sharp_\omega\alpha,\,(Y\in L,\alpha\in ann\,\bar{L}),\end{array}$$
where $X_L=X_+=pr_LX,\,X_{\bar{L}}=X_-=pr_{\bar{L}}
X$, and the Courant algebroid structure of $\mathbf{L}$ transfers to $TM$. The resulting Courant algebroid structure has anchor $pr_L$ and the following operations:
\begin{equation}\label{gLomega}
g_L(X,Y)=\frac{1}{2}(\omega(X_{\bar{L}},Y_L)
+\omega(Y_{\bar{L}},X_L))=(1/2)\gamma(X,Y), \end{equation}
\begin{equation}\label{partialf2}
\partial_{(L,\bar L)}f=\sharp_\omega(pr_{ann\,\bar{L}}df),
\end{equation}
\begin{equation}\label{Ccuomega}\begin{array}{rcl}
[X,Y]_{(L,\bar L)}&=&
[X_L,Y_L] +\sharp_\omega pr_{ann\,\bar{L}}[ \mathfrak{L}_{X_L}i(Y_{\bar L})
\omega\vspace*{2mm}\\ &-&\mathfrak{L}_{Y_L}i(X_{\bar L})\omega
+\frac{1}{2}d(\omega(X,Y))],\end{array}
\end{equation}
\begin{equation}\label{Dcuomega}\begin{array}{rcl}
X\bigstar_{(L,\bar L)}Y&=&
[X_L,Y_L] +\sharp_\omega pr_{ann\,\bar{L}}[ \mathfrak{L}_{X_L}i(Y_{\bar L})
\omega\vspace*{2mm}\\ &-&\mathfrak{L}_{Y_L}i(X_{\bar L})\omega
+d(\omega(X_{\bar L},Y_L))].\end{array}
\end{equation}

The classical commutation formula $$i([X,Y])=\mathfrak{L}_Xi(Y)-i(Y)\mathfrak{L}_X$$
changes the previous expressions into
\begin{equation}\label{Ccuomega1}\begin{array}{l}
[X,Y]_{(L,\bar L)}=
[X_L,Y_L] +pr_{\bar L}([X_L,Y_{\bar L}]+[X_{\bar L},Y_L])\vspace*{2mm}\\
+\sharp_\omega pr_{ann\,\bar{L}}[ i(Y_{\bar L})\mathfrak{L}_{X_L}\omega
 -i(X_{\bar L})\mathfrak{L}_{Y_L}\omega
+\frac{1}{2}d(\omega(X,Y))],\end{array}
\end{equation}
\begin{equation}\label{Dcuomega1}\begin{array}{l}
X\bigstar_{(L,\bar L)}Y=
[X_L,Y_L] +pr_{\bar L}([X_L,Y_{\bar L}]+[X_{\bar L},Y_L])\vspace*{2mm}\\
+\sharp_\omega pr_{ann\,\bar{L}}[ i(Y_{\bar L})\mathfrak{L}_{X_L}\omega
 -i(X_{\bar L})\mathfrak{L}_{Y_L}\omega
+d(\omega(X_{\bar L},Y_L))].\end{array}
\end{equation}

If we switch between $L$ and $\bar{L}$ and replace $\omega$ by $-\omega$, we similarly get a second Courant algebroid structure with anchor $pr_{\bar{L}}$, with the same metric $(1/2)\gamma$ and with the operations
$$
\partial_{(\bar L,L)}f=-\sharp_\omega(pr_{ann\,L}df),
$$
$$\begin{array}{l}
[X,Y]_{(\bar L,L)}=
[X_{\bar{L}},Y_{\bar{L}}] +pr_L([X_{\bar{L}},Y_L]+[X_L,Y_{\bar{L}}])\vspace*{2mm}\\
+\sharp_\omega pr_{ann\,L}[ i(Y_L)\mathfrak{L}_{X_{\bar{L}}}\omega
 -i(X_L)\mathfrak{L}_{Y_{\bar{L}}}\omega
+\frac{1}{2}d(\omega(X,Y))],\end{array}
$$
$$\begin{array}{l}
X\bigstar_{(\bar{L},L)}Y=
[X_{\bar{L}},Y_{\bar{L}}] +pr_L([X_{\bar{L}},Y_L]+[X_L,Y_{\bar{L}}])\vspace*{2mm}\\
+\sharp_\omega pr_{ann\,L}[ i(Y_L)\mathfrak{L}_{X_{\bar{L}}}\omega
 -i(X_L)\mathfrak{L}_{Y_{\bar{L}}}\omega
+d(\omega(X_L,Y_{\bar{L}}))].\end{array}
$$
\end{proof}

Notice the following consequence of the first formula (\ref{gLomega}):
\begin{equation}\label{bemol}
 \begin{array}{l}\flat_{g_L}|_{\bar L}=\frac{1}{2}\flat_{\omega}|_{\bar L},\, \flat_{g_L}|_L=-\frac{1}{2}\flat_{\omega}|_L,\vspace*{2mm}\\
\sharp_{g_L}|_{ann L}=-2\sharp_\omega|_{ann L},\, \sharp_{g_L}|_{ann \bar{L}}=2\sharp_\omega|_{ann \bar{L}}.\end{array}\end{equation}
\begin{corol}\label{corolsuma} The tangent bundle of the para-Hermitian manifold $(M,\gamma,F)$ is endowed with a natural structure of a metric algebroid with metric $\gamma$ and anchor equal to the identity map.\end{corol}
\begin{proof}
Put $\bigstar_\omega=\bigstar_{(L,\bar L)}+\bigstar_{(\bar L,L)}$, i.e., \begin{equation}\label{starM} \begin{array}{l} X\bigstar_\omega Y=[X,Y]+\sharp_\omega pr_{ann\,\bar{L}}[d(\omega(X_{\bar L},Y_L))]\vspace*{2mm}\\ -\sharp_\omega pr_{ann\,L}[d(\omega(Y_{\bar L},X_L))] +\sharp_\omega pr_{ann\,\bar{L}}[ i(Y_{\bar L})\mathfrak{L}_{X_L}\omega \vspace*{2mm}\\ -i(X_{\bar L})\mathfrak{L}_{Y_L}\omega] +\sharp_\omega pr_{ann\,L}[ i(Y_L)\mathfrak{L}_{X_{\bar{L}}}\omega -i(X_L)\mathfrak{L}_{Y_{\bar{L}}}\omega].\end{array}\end{equation}
Since the sum of the two anchors is $pr_L+pr_{ \bar{L}}=Id$, if we add the metric compatibility properties of the terms, we get metric compatibility of $\bigstar_\omega$.
Finally, from (\ref{starM}) and (\ref{bemol}) we get
$$ X\bigstar_\omega X=\partial(\gamma(X,X))=\frac{1}{2}grad(\gamma(X,X)),
$$ where $\partial$ is defined by the metric $\gamma$ and the identity anchor.
\end{proof}

With the usual passage from the Dorfman to the Courant bracket, we get	a metric bracket, equal to the sum of the Courant brackets of the two foliations:
\begin{equation}\label{crM} \begin{array}{l} [X,Y]_\omega=[X,Y]+\frac{1}{2}\sharp_\omega[d(\omega(X,Y))] +\sharp_\omega pr_{ann\,\bar{L}}[ i(Y_{\bar L})\mathfrak{L}_{X_L}\omega \vspace*{2mm}\\ -i(X_{\bar L})\mathfrak{L}_{Y_L}\omega] +\sharp_\omega pr_{ann\,L}[ i(Y_L)\mathfrak{L}_{X_{\bar{L}}}\omega -i(X_L)\mathfrak{L}_{Y_{\bar{L}}}\omega].\end{array}\end{equation}

The Leibniz property does not hold for the bracket $[\,,\,]_\omega$ since it would imply that $(TM,\gamma,Id,\bigstar_\omega)$ is a Courant algebroid and
it would follow that
$$Id(X\bigstar_\omega Y)=[Id(X),Id(Y)],$$
which is not true.

Another negative result is that, if $M$ is a para-Hermitian manifold, the subbundles $L,\bar{L}$, seen as dual to each other via pairing by $\omega$ and with the Lie algebroid structures given by the Lie bracket, do not define a Lie bialgebroid. Indeed, formulas (\ref{Ccuomega1}), (\ref{Dcuomega1}) can be put in a nice form using the identification of $L^*$ with $\bar L$ and the Lie algebroid operations
$$ d_Lf=\sharp_\omega pr_{ann\,\bar{L}}df
\stackrel{(\ref{partialf2})}{=}\partial f,
\;\;f\in C^\infty(M),$$
$$ \mathfrak{L}^L_S\bar{U}=pr_{ \bar{L}}[S,\bar{U}]
+\sharp_\omega pr_{ann\,\bar{L}}[i( \bar{U})\mathfrak{L}_S\omega]
\;\;S\in L,\bar{U}\in\bar{L},$$
where $ \mathfrak{L}^L$ is the Lie derivative of $L$.
With these operations we get
$$ \begin{array}{l}
[X,Y]_{(L,\bar L)}=[X_L,Y_L] + \mathfrak{L}^L_{X_L}Y_{\bar L}-\mathfrak{L}^L_{Y_L}X_{\bar L}
+\frac{1}{2}d_L(\omega(X,Y)),\vspace*{2mm}\\ X\bigstar_{(L,\bar L)}Y=
[X_L,Y_L]+ \mathfrak{L}^L_{X_L}Y_{\bar L}-\mathfrak{L}^L_{Y_L}X_{\bar L}
+d_L(\omega(X_{\bar L},Y_L)).\end{array}
$$
Similar expressions hold for the operations with the index $(\bar{L},L)$, and if we add the two brackets we get an expression of $[\,,\,]_\omega$ that coincides with the expression of the Courant bracket of a Lie bialgebroid given in \cite{LWX}. However, we are not in the case of a Lie bialgebroid since, otherwise, $TM=L\oplus\bar{L}$ would be a Courant algebroid with identity anchor, which is impossible. Indeed, for a Courant algebroid $\rho\circ\partial=0$ \cite{LWX} and on $TM$ with $\rho=Id$ we would get $df=0$, $\forall f\in C^\infty(M)$.

We will show that the $\omega$-bracket is related to a more general bracket that will be needed for the construction of the field's action.

For any (pseudo)Riemannian manifold $(M,\gamma)$, $(TM,g=\gamma,\rho=Id)$ (hence, $\partial f=(1/2)grad\,f$) has a natural structure of a metric algebroid with the operations defined by \cite{V0}
\begin{equation}\label{Ccroset} X\bigstar_\gamma Y=[X,Y]_\gamma+\frac{1}{2}grad(\gamma(X,Y)),\;
[X,Y]_\gamma=[ X,Y]-X\wedge_{\nabla^0} Y, \end{equation} where \begin{equation}\label{wedge0}\gamma(Z,X\wedge_{\nabla^0} Y)
=\frac{1}{2}[\gamma(X,\nabla^0_ZY)-
\gamma(Y,\nabla^0_ZX)])\end{equation}
and $\nabla^0$ is the Levi-Civita connection of $\gamma$.
The metric compatibility of $\bigstar_\gamma$ means that we have
$$ Z(\gamma(X,Y))=\gamma(Z\bigstar_\gamma X,Y)+
\gamma(X,Z\bigstar_\gamma Y).$$

Further computations lead to the following equivalent expressions
\begin{equation}\label{crg} \begin{array}{rcl}
[X,Y]_\gamma&=&\frac{1}{2}\{[X,Y]+\sharp_\gamma( \mathfrak{L}_X(\flat_\gamma Y)- \mathfrak{L}_Y(\flat_\gamma X))\}
\vspace*{2mm}\\ &=&\frac{3}{2}[X,Y] +\sharp_\gamma(\flat_{ \mathfrak{L}_X\gamma}Y-\flat_{ \mathfrak{L}_Y\gamma}X)\vspace*{2mm}\\ &=&\frac{1}{2}\{[X,Y]+\sharp_\gamma(i(X)d(\flat_\gamma Y)-i(Y)d(\flat_\gamma X))\}.
\end{array}\end{equation}
The third equality follows directly from the first and the Cartan relation between the Lie derivative and the exterior differential. The proof of the first equality (\ref{crg}) is by checking that it produces the same value of $\gamma([X,Y]_\gamma,Z)$ as the original expression (\ref{Ccroset}). Notice also that formula (\ref{Ccroset}) yields
$$ (grad\,f)\bigstar_\gamma Y=\nabla^0_{grad\, f}Y-grad\,(Yf).$$
The Leibniz property does not hold for $\bigstar_\gamma$ for the same reason as in the case of $\bigstar_\omega$.
\begin{prop}\label{cazparaK} On a para-K\"ahler manifold the brackets $[\,,\,]_\omega$, $[\,,\,]_\gamma$ coincide.\end{prop}
\begin{proof}
We prove that the formula
\begin{equation}\label{paraHcr} [X,Y]_\gamma=[X,Y]_\omega+\frac{1}{2}\sharp_\omega[i(Y)i(X)d\omega]
\end{equation}
holds on any para-Hermitian manifold.
This follows by evaluating	 $\omega([X,Y]_\gamma-[X,Y]_\omega,Z)$ on arguments of type $(X_L,Y_L)$, $(X_{\bar{L}},Y_{\bar{L}})$, $(X_L,Y_{\bar{L}})$ while using (\ref{crM}) and the third expression (\ref{crg}). We shall also use the decompositions $T^*M=ann\,L\oplus ann\,\bar{L}$, $d=d_L+d_{\bar{L}}$, where $d_L,d_{\bar{L}}$ are the differentials of the Lie algebroids $L,\bar{L}$, and formulas (\ref{bemol}). The conclusion of the proposition follows from (\ref{paraHcr}) because $d\omega=0$ in the para-K\"ahler case.\end{proof}

Formulas (\ref{paraHcr}) and (\ref{crg}) yield expressions of $[X,Y]_\omega$ that do not use the decomposition of $X,Y$ into the $L,\bar{L}$ components. For instance, if we combine the third expression (\ref{crg}) with (\ref{paraHcr}) and notice that (\ref{bemol}) implies
$$ \flat_{g_L}=-\frac{1}{2}\flat_\omega\circ F,\, \sharp_{g_L}=-2F\circ\sharp_\omega,$$ we obtain the formula
$$\begin{array}{r} [X,Y]_\omega=\frac{1}{2}[X,Y]+
\frac{1}{2}F\sharp_\omega[i(X)di(FY)\omega\vspace*{2mm}\\ -i(Y)di(FX)\omega]
-\frac{1}{2}\sharp_\omega[i(Y)i(X)d\omega].\end{array}$$

The $\gamma$-bracket is a generalization of the C-bracket of double field theory. Indeed, it was shown in \cite{V0} that the C-bracket defined on flat para-K\"ahler manifolds with local coordinates $(x^i,\tilde{x}_j)$ and with $\omega=dx^i\wedge d\tilde{x}_i$ is the unique metric bracket equal to $0$ on pairs of vectors $\partial/\partial x^i,
\partial/\partial \tilde{x}_j$. In the flat para-K\"ahler case, formula (\ref{paraHcr}) yields these values of the $\gamma$-bracket, precisely. Thus, the $\gamma$-bracket is the {\it para-Hermitian C-bracket}. (The same reason would allow us to give this name to the $\omega$-bracket, but, the latter is not convenient for the construction of the action.)

Finally, we notice that, in the general case, the $\gamma$-metric product does not yield generalized gauge transformations similar to those defined in the K\"ahler flat case \cite{V0}. Instead, we may consider the {\it partial gauge transformations} $$ \mathfrak{T}^L_XY=X\bigstar_{(L,\bar{L})}Y,\;
\mathfrak{T}^{\bar{L}}_XY=X\bigstar_{(\bar{L},L)}Y,$$ which may be extended to tensor fields as in the flat case (\cite{V0}, formula (3.7)). The conditions $$\begin{array}{l}\mathfrak{T}^L_X\mathfrak{T}^L_Y
-\mathfrak{T}^L_Y\mathfrak{T}^L_X=
\mathfrak{T}^L_{X\bigstar_{(L,\bar{L})}Y},\vspace*{2mm}\\ \mathfrak{T}^{\bar{}L}_X\mathfrak{T}^{\bar{L}}_Y
-\mathfrak{T}^{\bar{L}}_Y\mathfrak{T}^{\bar{L}}_X=
\mathfrak{T}^{\bar{L}}_{X\bigstar_{(\bar{L},L)}Y}\end{array}$$
hold since the Dorfman bracket of a Courant algebroid satisfies the Leibniz identity.
\section{Field connections and action}
It is natural to expect the action of a field to be the integral of a scalar curvature analog of a connection derived from the field given by the compatible metric $g$ of the para-Hermitian manifold $(M,F,\gamma)$ and which preserves the neutral metric $\gamma$. We see the last condition as T-duality in the present framework and we will indicate a process that gives a canonical connection of the required type.

The connections that preserve the metrics $\gamma,g$ will be called {\it double metric connections}; generally, the Levi-Civita connections of $\gamma$ and $g$ are not double metric connections.
\begin{prop}\label{bimcon} Double metric connections are in a one-to-one correspondence with the pairs of $k$-metric connections on $L$, where $(k,\beta)$ are the $L$-components of the field $g$.\end{prop}
\begin{proof} The $L$-components $(k,\beta)$ were defined in Definition \ref{def2}.
Let $\nabla$ be a double metric connection on $TM$. Then, $\nabla$ also preserves $H$ and it must have the form
$\nabla=\nabla^++\nabla^-$, where the terms are connections on $V_\pm$. Accordingly, we get well defined operators $D^\pm_XY$, $X\in\Gamma TM,Y\in\Gamma L$, that satisfy the condition
\begin{equation}\label{defD} \nabla_X(\iota_\pm Y)=\iota_\pm(D^\pm_XY)\end{equation} and $D^\pm$ are connections on the vector bundle $L$. Moreover, since $g|_{V_\pm}=k$, the restriction of the condition $\nabla_Xg=0$ to $V_\pm$ is equivalent to $D^\pm_Xk=0$. Conversely, if we have a pair $D^\pm$ of $k$-metric connections on $L$ and we define $\nabla$ by (\ref{defD}), $\nabla$ preserves $H$ (since it preserves the eigenbundles $V_\pm$ of $H$), it preserves $g$ (since it preserves the restrictions $g|_{V_\pm}$ as well as the orthogonality condition $g|_{V_+\times V_-}=0$) and it preserves $\gamma$ (because of (\ref{defk})).\end{proof}
\begin{defin}\label{type12} {\rm A double metric connection such that $D^+=D^-=D$ will be called of {\it type $1$}. A double metric connection such that $D^+\neq D^-$ will be called of {\it type $2$}.}\end{defin}
\begin{prop}\label{proptype1} Any double metric connection of type 1 preserves the subbundle $\bar{L}$. It also preserves the subbundle $L$ iff the corresponding connection $D$ of $L$ commutes with the endomorphism $\psi_+$ of {\rm(\ref{matriceaH})}.\end{prop}
\begin{proof}
Formula (\ref{eqiota}) implies that $\forall Y\in L$ one has \begin{equation}\label{eqauxc}\begin{array}{l}
\iota_+Y-\iota_-Y=2\tilde{\theta}^{-1}Y\in \bar{L},
\vspace*{2mm}\\ \iota_+Y+\iota_-Y=2Y-2\tilde{\theta}^{-1}\psi_+Y.\end{array}\end{equation}
Thus, every vector of $\bar{L}$ is in the image of $\iota_+-\iota_-$ and conversely, which implies the first assertion. More precisely, the first formula (\ref{eqauxc}) yields
$$ \nabla_XZ=\tilde{\theta}^{-1}D_X(\tilde{\theta} Z),
\hspace{3mm}Z\in\bar{L}.$$
Furthermore, (\ref{eqauxc}) implies $$ Y=\frac{1}{2}\{\iota_+(Y+\psi_+Y) +\iota_-(Y-\psi_+Y)\}$$ and using (\ref{defD}) we get
$$ \nabla_XY=D_XY+\tilde{\theta}^{-1}(D_X(\psi_+Y)- \psi_+(D_XY))\;\;(Y\in\Gamma L).$$ This result implies the second assertion.
\end{proof}

We will define a canonical $k$-metric connection on $L$ by using some old ideas from foliation theory \cite{VCz}. If $D$ is a connection on $L$, we define the torsion $\tau^D\in Hom(TM\otimes TM,L)$ by
$$ \tau^D(X,Y)= D_X(pr_L Y)-D_Y(pr_L X)
-pr_L[X,Y].$$
\begin{lemma}\label{lematau} There exists a unique connection $D$ on $L$ such that $k$ is preserved by parallel translation along curves tangent to $L$ and $\tau^D(X,Y)=0$ if at least one of the arguments is in $L$.\end{lemma}
\begin{proof} For $X\in \bar{L},Y\in L$, the vanishing of $\tau(X,Y)$ yields
$D_XY=pr_L[X,Y]$. Furthermore, the conditions $D_Zk(X,Y)=0,T^D(X,Y)=0$ for $Z,X,Y\in\Gamma L$ may be processed as in the case of a Riemannian connection, which yields a well defined expression of $D_XY$ for $X,Y\in\Gamma L$.\end{proof}

The resulting connection $D$ satisfies $T^D(X,Y)=0$ for any arguments; zero torsion for $X,Y\in\Gamma \bar{L}$ follows from the integrability of $\bar{L}$.

Now, for any connection $D$ on $L$, the formula
\begin{equation}\label{barD} D'_XY=D_XY+\frac{1}{2}\sharp_k\flat_{D_Xk}Y,
\hspace{5mm} X\in\Gamma TM,Y\in \Gamma L,\end{equation} yields a new connection such that $D'k=0$. In particular, the connection $D'$ associated to $D$ of Lemma \ref{lematau} is a canonical $k$-metric connection on $L$ such that $D'_XY=D_XY$ if $X,Y\in L$. The double metric connection $\nabla'$ associated to the pair $(D',D')$ will be called the {\it initial connection of type $1$}.

The construction of the initial connection of type 1 may be modified such as to also include the form component $\beta$. This modification is suggested by generalized geometry \cite{{G2},{V1}} and leads to an {\it initial connection of type} 2. The modification consists in replacing the condition $\tau^D(X,Y)=0$ of Lemma \ref{lematau} by the condition
\begin{equation}\label{bismut} \tau^D(X,Y)=\pm2\sharp_k [i(pr_LY)i(pr_LX)d_L\beta]. \end{equation}
Nothing changes if at least one argument is in $\bar{L}$ but half of the last term of (\ref{bismut}) is added to the covariant differential $D$ if $X,Y\in L$. In this way, and using again (\ref{barD}), we get a pair of $k$-metric connections $\hspace{1pt}^\beta\hspace{-2pt}D^\pm$ on $L$ and we define the initial connection of type 2 as the double metric connection $\hspace{1pt}^\beta\hspace{-1pt}\nabla$ on $M$ that is associated to this pair.

As in \cite{{G2},{V0}}, for a $\gamma$-metric connection $\nabla$ on a (pseudo-)Riemannian manifold $(M,\gamma)$,
we introduce an invariant, which we call the {\it $\gamma$-torsion} of $\nabla$, defined by
\begin{equation}\label{torsC} T^\nabla_\gamma(X,Y)=\nabla_XY-\nabla_YX-X\wedge_\nabla Y-[X,Y]_\gamma,\end{equation}
where $\wedge_\nabla$ is given by (\ref{wedge0}) with the Levi-Civita connection $\nabla^0$ replaced by $\nabla$. Without the term $-X\wedge_\nabla Y$ the result is not a tensor field, but, (\ref{propcrg}) shows that (\ref{torsC}) is a tensor field. Furthermore, using the metric compatibility of the $\gamma$-bracket, we can check the total skew-symmetry	of the {\it covariant $\gamma$-torsion} \cite{{G2},{V0}}
$$ \tau^\nabla_\gamma(X,Y,Z)=\gamma(T^\nabla_\gamma(X,Y),Z).$$ \begin{prop}\label{0gammators} The $\gamma$-metric connections with zero $\gamma$-torsion bijectively correspond to $3$-covariant tensor fields $\Xi$ on $M$ such that
\begin{equation}\label{condXi0}
\Xi(X,Y,Z)=-\Xi(X,Z,Y),\,\sum_{Cycl(X,Y,Z)}\Xi(X,Y,Z)=0.
\end{equation}\end{prop}
\begin{proof} Put $$\nabla_XY=\nabla^0_XY+\Theta(X,Y),$$ where the {\it covariant deformation tensor} $$\Xi(X,Y,Z)=\gamma(\Theta(X,Y),Z)$$ satisfies the property
\begin{equation}\label{Ximetric} \Xi(X,Y,Z)=-\Xi(X,Z,Y).\end{equation}
Then, we get
\begin{equation}\label{torsC2} \tau_\gamma^\nabla(X,Y,Z)=\gamma(T^\nabla(X,Y),Z)
+\Xi(Z,X,Y),\end{equation}
where $T^\nabla$ is the usual torsion of the connection. Since the Levi-Civita connection $\nabla^0$ has no torsion,
$$T^\nabla(X,Y)=\Theta(X,Y)-\Theta(Y,X)$$ and (\ref{Ximetric}) gives
\begin{equation}\label{torsC3} \tau^\nabla_\gamma(X,Y,Z)=\sum_{Cycl(X,Y,Z)}\Xi(X,Y,Z).\end{equation}
This result and (\ref{Ximetric}) justify the conclusion. \end{proof}

Notice that there exists a family of such tensor fields $\Xi$. Indeed, since alternation is an epimorphism $$alt:T^*M\otimes(\wedge^2T^*M)\rightarrow\wedge^3T^*M,$$ where the dimension of the first space is $m^2(m-1)/2$ and the dimension of the second is $m(m-1)(m-2)/6$, the dimension of the kernel is $m(m^2-1)/3$ ($m=dim\,M$).

If we come back to the definition of $\Xi$, and since the metric character of the connections ensures the first condition (\ref{condXi0}), we get
\begin{prop}\label{proptau0} The $\gamma$-metric connection $\nabla$ has a vanishing $\gamma$-torsion iff
\begin{equation}\label{tau0ciclic} \begin{array}{r}
\sum_{Cycl(X,Y,Z)}\gamma(\nabla_XY,Z)
=\sum_{Cycl(X,Y,Z)}\gamma(\nabla^0_XY,Z)\vspace*{2mm}\\ =\frac{1}{2}
\sum_{Cycl(X,Y,Z)}\{Z(\gamma(X,Y))+\gamma([X,Y],Z)\}.\end{array}
\end{equation}
\end{prop}
\begin{proof} The first equality (\ref{tau0ciclic}) is just the second condition (\ref{condXi0}). The second equality (\ref{tau0ciclic}) follows from the well known global expression of the Levi-Civita connection (e.g., \cite{KN}, Chapter IV, \&2).\end{proof}

Now, we return to the (pseudo-)Riemannian field $g$ on the para-Hermitian manifold $(M,\gamma,F)$ and to the double metric connections $\nabla$. For such a connection, the restriction of the $\gamma$-torsion to $V_+\times V_-$ will be called the {\it mixed $\gamma$-torsion} and we get the following result.
\begin{prop}\label{cutting} If $\nabla'$ is a double metric connection with a vanishing mixed $\gamma$-torsion, there exists a unique deformation tensor field $\Phi$ with a totally skew symmetric, corresponding, covariant deformation such that the new connection $\nabla=\nabla'+\Phi$ is double metric and has a vanishing $\gamma$-torsion.
\end{prop}
\begin{proof} Denote $\Psi(X,Y,Z)=\gamma(\Phi(X,Y),Z)$. Then, because the connections $\nabla,\nabla'$ preserve the metric $\gamma$, (\ref{Ximetric}) and (\ref{torsC2}) hold for $\Psi$ instead of $\Xi$ and (\ref{torsC3}) becomes
\begin{equation}\label{tauPsi}
\tau_\gamma^\nabla(X,Y,Z)=\tau_\gamma^{\nabla'}(X,Y,Z)+
\sum_{Cycl(X,Y,Z)}\Psi(X,Y,Z).\end{equation}
Furthermore, if $\nabla'$ preserves $g$ too, the condition for $\nabla$ to preserve $g$ is $\Phi(X,HY)=H\Phi(X,Y)$, which is equivalent to
$$ \Psi(X,HY,HZ)=\Psi(X,Y,Z),$$
hence, with
\begin{equation}\label{PsiH2} \Psi(X,S,U)=0,\;\;\forall S\in V_+,U\in V_-.\end{equation} Equality (\ref{tauPsi}) tells us that the required conclusion holds iff the deformation is given by
$$\Psi(X,Y,Z)=-(1/3)\tau_\gamma^{\nabla'}(X,Y,Z).$$ Since we require (\ref{PsiH2}), the indicated choice of $\Psi$ is valid iff
$$\tau_\gamma^{\nabla'}(X,S,U)=0,\;\;\forall S\in V_+,U\in V_-.$$
In view of the total skew symmetry of $\tau_\gamma^{\nabla'}$ this is equivalent with the vanishing of the mixed $\gamma$-torsion of $\nabla'$.
\end{proof}

We will say that $\nabla$ is obtained from $\nabla'$ by removing the $\gamma$-torsion.

The vanishing of the mixed $\gamma$-torsion is equivalent with the pair of conditions
$$ \begin{array}{c}
\gamma(T^{\nabla'}_\gamma(S,U),Z_{V_+})=0,\;
\gamma(T^{\nabla'}_\gamma(S,U),Z_{V_-})=0,\vspace*{2mm}\\ \forall S,Z_{V_+}\in V_+,\,
U,Z_{V_-}\in V_-.\end{array}$$ Then, if we insert the expression of $T^{\nabla'}_\gamma$ and use the preservation of the two $\gamma$-orthogonal subbundles $V_\pm$, we will see that the vanishing of the mixed torsion is equivalent to
\begin{equation}\label{mixt02}
\nabla'_{U}S=pr_{V_+}[U,S]_\gamma,\; \nabla'_{S}U
=pr_{V_-}[S,U]_\gamma.\end{equation}
Thus, these covariant derivatives are the same for all double metric connections with a vanishing mixed $\gamma$-torsion.
Putting $U=\iota_-X,S=\iota_+Y$ where $X,Y\in L$ and using (\ref{defD}), we see that conditions (\ref{mixt02}) are equivalent with
\begin{equation}\label{mixt03} \begin{array}{l} D^+_{\iota_-X}Y=pr_Lpr_{V^+}[\iota_-X,\iota_+Y]_\gamma,
\vspace*{2mm}\\ D^-_{\iota_+Y}X=pr_Lpr_{V^-}[\iota_+Y,\iota_-X]_\gamma.
\end{array}\end{equation}

Using (\ref{lincroset}) we see that the expressions (\ref{mixt03})are compatible with the definition of a connection on $L$.
Therefore, if we add operators $D^-_{\iota_-X}Y,D^+_{\iota_+X}Y$ to (\ref{mixt03}) to define connections on $L$, we may, afterwards, continue as earlier and get a double metric connection with a vanishing mixed $\gamma$-torsion.

In particular, we get a canonical connection if we define
$$
D^-_{\iota_-X}Y=\hspace{1pt}^\beta\hspace{-2pt}
D^+_{\iota_-X}Y,\:D^+_{\iota_+X}Y=
\hspace{1pt}^\beta\hspace{-2pt}D^+_{\iota_+X}Y,$$
where $\hspace{1pt}^\beta\hspace{-2pt}D^\pm$ are the $k$-metric connections on $L$ that were used in the definition of the initial connection of type 2. (The formulas would simplify by using in both cases the connection $D'$ used for the initial connection of type 1, but, our choice has the advantage of introducing the form $\beta$ into the connection.)	Then, the connections $D^\pm$ are given by
\begin{equation}\label{LC+}
D^+_ZY=\hspace{1pt}^\beta\hspace{-2pt}D^\pm_{pr_{V_\pm}Z}Y+
pr_Lpr_{V_\pm}[pr_{V_\mp}Z,\iota_\pm Y]_\gamma,\;Y\in L,Z\in TM\end{equation}
and they preserve the metric $k$. Indeed, $\forall X,Y\in L,Z\in TM$, using the fact that $\hspace{1pt}^\beta\hspace{-2pt}D^\pm$ preserve $k$, the metric property of the $\gamma$-bracket and the $\gamma$-orthogonality of $V_\pm$, we have:
$$\begin{array}{l} k(D^\pm_ZX,Y)+k(X,D^\pm_ZY)=k(\hspace{1pt}^\beta\hspace{-2pt}D^\pm_ZX,Y)+
k(X,\hspace{1pt}^\beta\hspace{-2pt}D^\pm_ZY)
\vspace*{2mm}\\ +\gamma(pr_{V_\pm}[pr_{V_\mp}Z,\iota_\pm X]_\gamma,\iota_\pm Y)+\gamma(X,pr_{V_\pm}[pr_{V_\mp}Z,\iota_\pm Y]_\gamma)
\vspace*{2mm}\\ =(pr_{V_\pm}Z)(k(X,Y))+(pr_{V_\mp}Z)(k(X,Y))
=Z(k(X,Y)).\end{array}$$

The double metric connection that corresponds to the pair of connections given by (\ref{LC+}) will be denoted by $\hspace{1pt}^\beta\hspace{-1pt}\nabla^\gamma$; it is a connection of type 2.
\begin{defin}\label{conexc} {\rm The connection obtained by removing the $\gamma$-torsion of $\hspace{1pt}^\beta\hspace{-1pt}\nabla^\gamma$ will be denoted by $\nabla^c$ and it will be called the {\it canonical double metric connection} of the field.}\end{defin}

For any double metric connection $\nabla$, we can define an analog of scalar curvature in the following way. We start with the usual curvature tensor $R^\nabla$ of the canonical connection and modify it into a tensor that is compatible with the decomposition $TM=V_+\oplus V_-$. Namely, we define the {\it modified canonical curvature operator} by
\begin{equation}\label{modifcurb} \tilde{R}^\nabla(X,Y)=R^\nabla(pr_{V_+}X,pr_{V_+}Y)
+R^\nabla(pr_{V_-}X,pr_{V_-}Y),\;\;X,Y\in TM.\end{equation}
Accordingly, we have $\tilde{R}^\nabla(X_+,Y_-)=0$, $\forall X_+\in V_+,Y_-\in V_-$. Furthermore, formula (\ref{defD}) and the definition of curvature yield
\begin{equation}\label{curbpm} \tilde{R}^\nabla(X,Y)(\iota_\pm U)=R^\nabla(X,Y)(\iota_\pm U) =\iota_\pm(R^{D^\pm}(X,Y)U),\end{equation} where $X,Y\in V_\pm,\,U\in L$
(in (\ref{curbpm}), the second equality holds $\forall X,Y\in TM$).

Now, assume that $\gamma|_{V_+}$ has the $\pm$-inertia indices $(p,q)$, $p+q=n$, which implies that $\gamma|_{V_-}$ has indices $(q,p)$, hence, by (\ref{ggamma}), $g|_{V_\pm}$ has the same indices $(p,q)$ for both $V_+$ and $V_-$. A local basis $(e_i,e_u,f_i,f_u)$ ($i=1,...,p,\,u=p+1,...,n$) of $TM$ is a {\it double (pseudo-)orthonormal basis} if  $e_i,e_u\in V_+,f_i,f_u\in V_-$ and
$$\begin{array}{l}
g(e_i,e_j)=\gamma(e_i,e_j)=\delta_{ij},g(e_i,e_u)=\gamma(e_i,e_u)=0,  \vspace*{2mm}\\ g(e_u,e_v)=\gamma(e_u,e_v)=-\delta_{uv},
\vspace*{2mm}\\ g(f_i,f_j)=-\gamma(f_i,f_j)=\delta_{ij},g(f_i,f_u)=-\gamma(f_i,f_u)=0,  \vspace*{2mm}\\ g(f_u,f_v)=-\gamma(f_u,f_v)=-\delta_{uv}.
\end{array}$$

Then, we shall define the {\it modified Ricci curvature} of $\nabla$ by
$$ \begin{array}{l}
\tilde{r}^\nabla(X,Y)=\frac{1}{2}\{\sum_{i=1}^p[g(e_i,\tilde{R}^\nabla(X,e_i)Y)+
g(e_i,\tilde{R}^\nabla(Y,e_i)X)]\vspace*{2mm}\\ - \sum_{u=p+1}^n[g(e_u,\tilde{R}^\nabla(X,e_u)Y)+
g(e_u,\tilde{R}^\nabla(Y,e_u)X)]\vspace*{2mm}\\ + \sum_{i=1}^p[g(f_i,\tilde{R}^\nabla(X,f_i)Y)+
g(f_i,\tilde{R}^\nabla(Y,f_i)X)]\vspace*{2mm}\\ - \sum_{u=p+1}^n[g(f_u,\tilde{R}^\nabla(X,f_u)Y)+
g(f_i,\tilde{R}^\nabla(Y,f_i)X)]\}.\end{array}$$
The result is invariant under changes of double (pseudo-)orthonormal bases since such changes belong to $O(p,q)\times O(p,q)$. The modified Ricci curvature is symmetric and it vanishes if evaluated on arguments in different subbundles $V_\pm$, hence it is a tensor that is compatible with the decomposition $TM=V_+\oplus V_-$.

Finally, we will define the {\it modified scalar curvature} by
\begin{equation}\label{scalarg} \tilde{\sigma}^\nabla=\sum_{i=1}^p\tilde{r}^\nabla(e_i,e_i)
-\sum_{u=p+1}^n\tilde{r}^\nabla(e_u,e_u) +\sum_{i=1}^p\tilde{r}^\nabla(f_i,f_i)
-\sum_{u=p+1}^n\tilde{r}^\nabla(f_u,f_u),\end{equation}
which is invariant under base changes for the same reason as $\tilde{r}^\nabla$.

The modified Ricci and scalar curvatures of the canonical double metric connection $\nabla^c$ are invariants of the field $g$. Therefore, we may define the action of the field as follows.
\begin{defin}\label{defact} {\rm The integral
\begin{equation}\label{action} \mathcal{A}(g)=\int_Me^{-2\phi}\tilde{\sigma}(g)
\sqrt{|det(g)|}dx^1\wedge...\wedge dx^{2m},
\end{equation} where $(x^1,...,x^{2m})$ are positively oriented, local coordinates on $M$, $\phi\in C^\infty(M)$ is the {\it dilation scalar} of the field \cite{HHZ} and $\tilde{\sigma}(g)=\tilde{\sigma}^{\nabla^c}$,
is the {\it action} of the field $g$ of $L$-components $(k,\beta)$.}
\end{defin}

Notice that a para-Hermitian manifold $M^{2m}$ is oriented since it has the nowhere vanishing form $\omega^m$. But, conditions ensuring that the integral (\ref{action}) is finite have to be required.
\section{Reduction}
In this section we discuss reduction of a field under the action of a symmetry group. We assume that the reader is familiar with reduction of symplectic manifolds e.g., \cite{{LM},{OR}}. For K\"ahler reduction we refer to \cite{HKLR} and for para-K\"ahler reduction to \cite{K}. We will review the reduction process in a form that is suitable for the reduction of a field and we will define reduction of a para-Hermitian manifold.

Before refereeing to symmetry groups, we consider {\it geometric reduction}, i.e., reduction to the quotient space of a manifold by a convenient foliation. We shall assume that the space of leaves of the latter is a manifold since we want the result of reduction to be in the $C^\infty$-category.
\begin{prop}\label{geomredth} Let $(N,\omega)$ be an almost presymplectic manifold, where $dim\,N=n$ and $rank\,\omega=2r\leq n$. Assume that {\rm1)} $K=ann\,\omega$ is tangent to a foliation $\mathcal{K}$, {\rm2)} locally, $K$ is spanned by infinitesimal automorphisms of $\omega$, {\rm3)} there exists a fibration $p:N\rightarrow Q$, where $Q$ is a manifold of dimension $2r$ and $p$ is constant on the leaves of $\mathcal{K}$. Then, $\omega=p^*\varpi$ where $\varpi$ is a well defined, non degenerate $2$-form on $Q$. Furthermore, if $g$ is a (pseudo-) Riemannian metric on $N$ with a non degenerate restriction $g_K$ and such that the restriction $g_{K'}$  ($K'=K^{\perp_{g}}$) is of the form $g_{K'}=p^*\lambda$, where $\lambda$ is a (pseudo-)Riemannian metric on $Q$, then, the Levi-Civita connection $\nabla^g$ of $g$ projects to the Levi-Civita connection $\nabla^\lambda$ of $\lambda$ on $Q$.
\end{prop}
\begin{proof}
As in foliation theory \cite{Mol}, objects on $N$ that either project to or are a lift of an object on $Q$ will be called {\it projectable} or {\it foliated}. The first required conclusion includes the fact that $\omega$ is a foliated form, which is equivalent to the conditions
$$i(X)\omega=0,\;\mathfrak{L}_X\omega=0,\;\; \forall X\in\Gamma K.$$ The first condition holds obviously. The second condition is implied by hypothesis 2) that gives a local expression $X=\sum f_iX_i$, where the sum is finite, $X_i\in\Gamma K$ and $ \mathfrak{L}_{X_i}\omega=0$. Thus, there exists a $2$-form $\varpi\in\Omega^2(Q)$ such that $\omega=p^*\varpi$. The non degeneracy of $\varpi$ follows from the definition of $K$.

It follows easily that $K$ is a foliation iff
$$i(X\wedge Y)d\omega=0,\;\;\forall X,Y\in K,$$
and the presymplectic condition $d\omega=0$ implies conditions 1) and 2). If $p$ has connected fibers, $Q$	 is the space of leaves $N/\mathcal{K}$.

In the second part of the proposition, the hypotheses on $g$ mean that $g$ is a foliated (pseudo-)Riemannian metric. The conclusion means that,
for every two foliated vector fields $X,Y\in K'$, $pr_{K'}\nabla^{g}_XY$ is  a foliated vector field that projects to $\nabla^\lambda_{[X]}[Y]$, where the brackets denote the projected vector fields on $Q$. The projectability of $pr_{K'}\nabla^{g}_XY$ is equivalent to the fact that
$$g(pr_{K'}\nabla^{g}_XY,V)=g(\nabla^{g}_XY,V)$$ is a foliated function for every foliated vector field $V\in\Gamma K'$. This, indeed, follows directly from the well known global expression of the Levi-Civita connection (\cite{KN}, Chapter IV, \&2). The same expression also shows that the projection to $Q$ is $\nabla^\lambda$.
\end{proof}
\begin{defin}\label{defred} {\rm Under the conditions of Proposition \ref{geomredth}, the manifold $(Q,\varpi)$ is the {\it reduction} of $(M,\omega)$ and, if the metric $g$ exists, $\lambda$ is the {\it reduction} of the metric $g$.}\end{defin}

The main applications of reduction are for submanifolds $\iota:N^n\hookrightarrow M$ of an almost symplectic manifold $(M^{2m},\omega)$ that have the constant rank $rank\,(\iota^*\omega)=2r$. Then, $(N,\omega_N=\iota^*\omega)$ is presymplectic with annihilator
$$K=ann\,\omega_N=TN\cap T^{\perp_\omega}N$$
and if the requirements for reduction are satisfied we get a reduced manifold $(Q,\varpi)$, which is said to be the {\it reduction of $(M,\omega)$ via $N$}.
Furthermore, if $M$ is also endowed with a (pseudo-)Riemannian metric $g$ such that the pullbacks $g_N,g_K$ are non-degenerate and the pullback $g_{K'}$ ($TN=K\oplus K'$, $K'=K^{\perp_{g_N}}$) is $\mathcal{K}$-projectable, then, $Q$ has a reduced (pseudo-)Riemannian metric $\lambda$, which is the projection of $g_{K'}$ and the corresponding Levi-Civita connections are related by
$$
\nabla^\lambda_{[X]}[Y]=[pr_{K'}\nabla^{g_N}_XY]=[pr_{K'}\nabla^{g}_XY],
$$ where $X,Y\in\Gamma TN$ are $\mathcal{K}$-projectable and the brackets denote projection to $Q$.
\begin{rem}\label{obscoiso} {\rm A coisotropic submanifold N satisfies the condition $T^{\perp_\omega}N\subseteq TN$, hence, it necessarily has a constant rank and may admit reduction.}\end{rem}

Proposition \ref{geomredth} may be used to derive the following result.
\begin{prop}\label{redKpK} Assume that the almost symplectic manifold $(M^{2n},\omega)$ is endowed with a (pseudo-)Riemannian metric $\gamma$ that satisfies one of the following two conditions
\begin{equation}\label{condKpK} \sharp_\omega\circ\flat_\gamma
=\mp\sharp_\gamma\circ\flat_\omega.\end{equation}
Then, for the minus sign, $J=\sharp_\omega\circ\flat_\gamma$ is an almost complex structure and $(M,J,\gamma)$ is a (pseudo-)Hermitian manifold with the fundamental form $\omega$ and for the plus sign, $F=\sharp_\omega\circ\flat_\gamma$ is an almost para-complex structure and $(M,F,\gamma)$ is a para-Hermitian manifold with the fundamental form $\omega$.

Let $\iota:N\rightarrow M$ be a submanifold of constant rank.
Assume that the subbundle $K=TN\cap T^{\perp_\omega}N\subseteq TN$ satisfies hypotheses {\rm1), 2), 3)} of Proposition \ref{geomredth} along the almost presymplectic manifold $(N,\iota^*\omega)$. Assume also that the metric $\gamma$ has non degenerate restrictions $g_N,g_K$ and a restriction $g_{K'}=p^*\lambda$ where $K'=TN\cap K^{\perp_\gamma}$ and $\lambda$ is a (pseudo-)Riemannian metric on $Q$. Finally, assume that
\begin{equation}\label{condK'}
K^{'\perp_\gamma}=K^{'\perp_\omega}.\end{equation}

Then, the corresponding reduced manifold $(Q,\varpi,\lambda)$ also is an almost (pseudo) Hermitian, respectively, an almost para-Hermitian manifold. Moreover, if the manifold $M$ is K\"ahler, respectively, para-K\"ahler, the same property holds for the reduced manifold $Q$.\end{prop}
\begin{proof} The conditions imposed on $\gamma$ show that $\iota^*\gamma$ satisfies the hypotheses required for the metric $g$ of Proposition \ref{geomredth} and $\lambda$ is the reduction of the metric $\gamma$ to $Q$. Moreover, condition (\ref{condK'}) is equivalent with the fact that $K'$ is invariant by $J$, respectively, $F$. Now, condition (\ref{condKpK}) is equivalent with $$\flat_\gamma
=\mp\flat_\omega\circ\sharp_\gamma\circ\flat_\omega.$$ If this condition is applied to vectors $X,Y\in K'$, then, the definitions of $\varpi,\lambda$ together with (\ref{condK'}) imply
$$\flat_\lambda
=\mp\flat_\varpi\circ\sharp_\lambda\circ\flat_\varpi.$$
The latter result and Proposition \ref{geomredth} justify the first conclusion of the present proposition.

For the second conclusion we have to check the condition $\nabla^\lambda\varpi=0$ \cite{AMT}. Equivalently, we have to check that
$\nabla^{\iota^*\gamma}_Z\omega(X,Y)=0$, for $K$-foliated vector fields $X,Y,Z\in K'$. This follows by the computation below, which holds for obvious reasons, including the fact that the second fundamental form of $N$ in $(M,\gamma)$ is normal to $N$:
$$\begin{array}{lcl}
\nabla^{\iota^*\gamma}_Z\omega(X,Y)&=&
Z(\omega(X,Y))-\omega(\nabla^{\iota^*\gamma}_ZX,Y)-
\omega(X,\nabla^{\iota^*\gamma}_ZY)\vspace*{2mm}\\
&=&Z(\omega(X,Y))-\gamma(pr_{K'}\nabla^{\iota^*\gamma}_ZX,JY)+
\gamma(JX,pr_{K'}\nabla^{\iota^*\gamma}_ZY)\vspace*{2mm}\\
&=&Z(\omega(X,Y))+\gamma(\nabla^{\iota^*\gamma}_ZX,JY)-
\gamma(JX,\nabla^{\iota^*\gamma}_ZY)\vspace*{2mm}\\
&=&Z(\omega(X,Y))+\gamma(\nabla^{\gamma}_ZX,JY)-
\gamma(JX,\nabla^{\gamma}_ZY)\vspace*{2mm}\\ &=&\nabla^\gamma_Z\omega(X,Y)\end{array}$$
In the para-Hermitian case, $J$ will be replaced by $F$.
The last covariant derivative vanishes for K\"ahler, respectively, para-K\"ahler manifolds.\end{proof}
\begin{rem}\label{obspropred} {\rm The reduced almost complex, respectively para-complex structure $\mathcal{J},\mathcal{F}$ defined on $Q$ by $\sharp_\varpi\circ\flat_\lambda$ is the projection by $p$ of the restriction of $J,F$ to $K'$. This implies that, if $J,F$ are integrable the same holds for $\mathcal{J},\mathcal{F}$. Another important fact is that condition (\ref{condK'}) necessarily holds if $N$ is a coisotropic submanifold of $(M,\omega)$. Indeed, then, $J(TN)$ also is a coisotropic subbundle of $TM$ with the $\omega$-orthogonal isotropic subbundle $JK$ and we get
$$K'=TN\cap K^{\perp_\gamma}=TN\cap (JK)^{\perp_\omega}=TN\cap(JTN),$$ which is a $J$-invariant subbundle of $TM$.}\end{rem}
\begin{defin}\label{Nreducing} {\rm A submanifold $N$ such that all the hypotheses of Proposition \ref{redKpK} hold will be called a {\it reducing submanifold}.}\end{defin}

Proposition \ref{redKpK}  leads to the following reduction theorem (notation is the same as in Proposition \ref{redKpK}).
\begin{prop}\label{redcamp} Let $(M,F,\gamma)$ be a para-Hermitian manifold and let $\iota:N\rightarrow M$ be a reducing submanifold. Let $g$ be a compatible metric that defines a field on $M$ and is such that: a) $g$ is non degenerate on $N$ and on $K$, b) $K^{\perp_g}\cap TN=K'$, c) $g|_{K^{\perp_{\iota^*g}}}=p^*\mu$ where $\mu$ is a (pseudo-)Riemannian metric on $Q$. Then, if $(Q,\varpi,\lambda)$ is the reduction of $M$ via $N$, the metric $\mu$ is compatible with the metric $\lambda$ and it defines the reduction of the field $g$.\end{prop}
\begin{proof} The subbundles $K,K'$ coincide with those defined in Proposition \ref{geomredth}. The fact that $(Q,\varpi,\lambda)$ is a para-Hermitian manifold was proven in Proposition \ref{redKpK} and Remark \ref{obspropred}. The existence of $\mu$ follows from the second part of Proposition \ref{geomredth}. Hypothesis b) implies that the compatibility between $\lambda$ and $\mu$ is equivalent with the restriction of the compatibility between $\gamma$ and $g$ to the subbundle $K'$.\end{proof}

We will use the previous results in order to define reduction by a symmetry group. We begin by the following general considerations, which we, improperly but conveniently, call an example.
\begin{example}\label{excuG} {\rm
Let $(M,\omega)$ be an almost symplectic manifold endowed with an $\omega$-preserving action of a Lie group $G$. Let $\iota:N\hookrightarrow M$ be a submanifold, which is invariant by a subgroup $G'\subseteq G$ and satisfies the following conditions: i) $N$ is $\omega$-orthogonal to and has a clean intersection with the orbits $G(x)$, $x\in N$, ii) $\forall x\in N$, the $G'$-orbit of $x$ is $G'(x)=N\cap G(x)$ and it is a maximally isotropic submanifold of $(N,\iota^*\omega)$, iii) the action of $G'$ on $N$ is free and proper. Conditions i)-iii) imply that, $\forall x\in N$, $ann(\iota^*\omega_x)=T(G'(x))$ and that $N$ is a submanifold of constant rank.

Coming back to the previous general notation, now, $\mathcal{K}$ is the foliation of $N$ by the connected components of the orbits of $G'$ and obviously satisfies the hypotheses 1)-3) of Proposition \ref{geomredth}. Therefore, there exists a reduced manifold $(Q,\varpi)$ of $(M,\omega)$ via $N$.

Then, let $\gamma$ be a $G$-invariant (pseudo-)Riemannian metric on $M$ such that the pullbacks $\gamma_N,\gamma_K$ are non degenerate. The $G$-invariance of $\gamma$ implies $\mathfrak{L}_{\xi_M}\gamma=0$ for the infinitesimal transformations $\xi_M$  defined by $\xi\in \mathfrak{g}$, where $\mathfrak{g}$ is the Lie algebra of $G$. Accordingly, since any vector field $Z\in\Gamma TN$ that is tangent to the orbits of $G'$ is locally spanned by infinitesimal transformations of $G'\subseteq G$, if $X,Y\in K'=K^{\perp_{\iota^*\gamma}}$ are foliated vector fields, we get
$$\mathfrak{L}_Z\gamma(X,Y)=Z(g(X,Y))=0,$$ hence, $\gamma_{K'}$ is projectable and $\gamma$ reduces to a metric $\lambda$ on $Q$.

If $\gamma$ satisfies the hypotheses of Proposition \ref{redKpK} and of Proposition \ref{redcamp} we obtain a reduction for para-Hermitian manifolds and fields via a group of symmetries.}
\end{example}

From symplectic geometry, it is known that situations of the type described in Example \ref{excuG} appear in the case of Hamiltonian actions with a momentum map and the corresponding process is {\it Marsden-Weinstein reduction}. We will formally extend the definitions to almost symplectic manifolds $(M,\omega)$, which requires for some adjustments as follows \cite{{FS},{Vas}}.

A {\it locally Hamiltonian vector field} $X$ will be required to satisfy the conditions
\begin{equation}\label{locHam} di(X)\omega=0,\; \mathfrak{L}_X\omega=0\end{equation} and it is {\it Hamiltonian} if
\begin{equation}\label{campH}
\exists f\in C^\infty(M),\; i(X)\omega=df,\; \mathfrak{L}_X\omega=0\end{equation} ((\ref{campH}) implies (\ref{locHam})). If (\ref{campH}) holds, $f$ is a {\it Hamiltonian function} of $X$, it is defined up to a constant, and we denote $X=X_f$. The set of all Hamiltonian functions on $M$ will be denoted by $ \mathcal{H}(M,\omega)$. It follows easily that
$X_{fh}=fX_h+hX_f$ ($f,h\in C^\infty(M)$).

For two locally Hamiltonian vector fields $X,Y$ one has
$$i([X,Y])\omega=d(\omega(Y,X)),$$ which yields the following {\it Poisson bracket} of Hamiltonian functions $\{f,h\}=\omega(X_h,X_f)$. The bracket has the Hamiltonian field $[X_f,X_h]$. The Poisson algebra $\mathcal{H}(M,\omega)$ is equal to $C^\infty(M)$ in the symplectic case, but, it may be much smaller in other cases.
\begin{example}\label{exaprHam} {\rm
Take
$$M=\{(x^1,x^2,x^3,x^4)\,/\,
x^1>0,x^2>0\}\subseteq\mathds{R}^4$$ and
$$\omega=x^1dx^2\wedge dx^3+x^2dx^1\wedge dx^4,$$
$X$ is locally Hamiltonian iff
$$X=\varphi(t)\left(\frac{\partial}{\partial x^3}+\frac{\partial}{\partial x^4}\right),
\;\;t=x^1\cdot x^2$$ and the corresponding Hamiltonian function is the primitive of $-\varphi$. Thus, $\mathcal{H}(M,\omega)=\{f(t)\}$,
$$X_f=-\frac{\partial f}{\partial t}\left(\frac{\partial}{\partial x^3}+\frac{\partial}{\partial x^4}\right)$$ and all the Poisson brackets are zero.}\end{example}

Formally, the definition of {\it Hamiltonian actions} and {\it equivariant momentum maps} will be the same as in symplectic geometry \cite{{LM},{OR}}, but, (\ref{campH}) will be taken into account. Let us assume the existence of a Hamiltonian action of the Lie group $G$ on the almost symplectic manifold $(M,\omega)$ and of an equivariant momentum map $\mathfrak{J}:M\rightarrow\mathfrak{g}^*$ from $M$ to the dual of the Lie algebra of $G$. We recover a situation of the type described in Example \ref{excuG} in the following way. Take $N= \mathfrak{J}^{-1}(\theta)$ where $\theta\in\mathfrak{g}^*$ is a non critical value of $\mathfrak{J}$. The submanifold $N$ is invariant by the subgroup $G'=G_\theta$, the isotropy subgroup of $\theta$ by the coadjoint action of $G$ on $ \mathfrak{g}^*$ and $N\cap G(x)=G_\theta(x)$, where $x\in N$, $G(x)$ denotes the $G$-orbit of $x$ and $G_\theta(x)$ denotes the $G_\theta$-orbit of $x$. Then, $\forall x\in N$, $T_xN\perp_\omega T_x(G(x))$ (this is an algebraic fact \cite{LM}, hence, it holds in the almost symplectic case too), therefore, $K=ann(\iota^*\omega_x)=T_xN\cap T_x(G(x))=T_x(G_\theta(x))$, and properties i), ii) of Example \ref{excuG} hold. Thus, we get the following result.
\begin{prop}\label{redMW} {\rm1.} Assume that the Lie group $G$ has a Hamiltonian action on the almost symplectic manifold $(M,\omega)$ with equivariant momentum map $\mathfrak{J}$. Let $N= \mathfrak{J}^{-1}(\theta)$, where $\theta\in\mathfrak{g}^*$ is a non critical value of $\mathfrak{J}$ and assume that the action of $G_\theta$ on $N$ is free and proper. Then, there exists a reduced almost symplectic manifold $(Q,\varpi)$ of $(M,\omega)$ via $N$. {\rm2.} Let $\gamma$ be a $G$-invariant metric on $M$, which is non degenerate on $N=\mathfrak{J}^{-1}(\theta)$ and on the orbits of $G_\theta$ and satisfies the conditions {\rm(\ref{condKpK}), (\ref{condK'})}
($K=TG_\theta(x)$, $x\in N$).
Then, $Q$ also has a reduced metric $\lambda$ and $(Q,\varpi,\lambda)$ is almost (pseudo-)Hermitian or para-Hermitian, according to the sign in {\rm(\ref{condKpK})}. {\rm3.} In the integrable, para-Hermitian case, consider a field given by a compatible metric $g$ on $M$, which is $G$-invariant, has non degenerate restrictions to $N$ and to the orbits of $G_\theta$ and such that the orbits $G_\theta(x)$, $x\in N$ have the same normal bundle for the two metrics $\iota^*\gamma,\iota^*g$. Then, the reduced manifold $(Q,\varpi,\lambda)$ is endowed with a reduced field given by a reduced, compatible metric $\mu$.\end{prop}
\begin{proof} Conveniently put together the results of Propositions
\ref{geomredth}, \ref{redKpK}, \ref{redcamp}. The missing detail is the projectability of the normal components of the metrics $\iota^*\gamma,\iota^*g$ to $Q$. But, this is an obvious consequence of the $G$-invariance of the metrics $\gamma,g$.
\end{proof}
\section{Examples}
We will define compatible metrics on several well-known para-Hermitian manifolds, whose definition we recall.
\begin{example}\label{GG} {\rm\cite{EST} Let $G$ be a connected Lie group endowed with a left invariant (pseudo-)Riemannian metric $\mu$, and let $(X_i)$, $(\omega^i)$ be a basis of left invariant vector fields and the dual basis of left invariant $1$-forms, respectively. Put $M=G_1\times G_2$ where $G_1= G_2=G$, $\iota_1(g)=(g,e),\iota_2(g)=(e,g)$, where $e$ is the unit of $G$, $\pi_1(g_1,g_2)=g_1,\pi_2(g_1,g_2)=g_2$ and, generally, attach the index $1,2$ to images by $\iota_1,\pi_1,\iota_2,\pi_2$, of objects of $G$. $M$ has two almost para-complex structures given by
$$FX_1=X_1,\,FX_2=-X_2,\,HX_1=X_2,\,
HX_2=X_1,\;X\in TG,$$
the non degenerate, neutral metric
$$\gamma=\mu_{ij}(\omega_1^i\otimes\omega_2^j+\omega_2^i\otimes\omega_1^j)
\;\;(\mu=\mu_{ij}\omega^i\otimes\omega^j),$$
the (pseudo-)Riemannian metric $$g=\mu_{ij}(\omega_1^i\otimes\omega_1^j+\omega_2^i\otimes\omega_2^j)$$ and the $2$-form
$$\omega=\mu_{ij}\omega_1^i\wedge\omega_2^j.$$
It is easy to see that $(M,F,\gamma)$ is a para-Hermitian manifold with the fundamental form $-\omega$ and that $g$ is a compatible (pseudo-)Riemannian metric with corresponding almost para-complex structure $H$. Using the Cartan equations $$d\omega^i=\frac{1}{2}(c^i_{jk}\omega^k\wedge\omega^j)$$ where $c^i_{jk}$ are the structure constants of the Lie algebra $\mathfrak{g}$ of $G$, it follows that $M$ is para-K\"ahler ($d\omega=0$) iff the group $G$ is Abelian, e.g., a torus.}\end{example}
\begin{example}\label{fibrtg} {\rm Let $M=TN$ be the total space of the tangent bundle of an $n$-dimensional, (pseudo-)Riemannian manifold $N$ with the metric $\mu$. $M$ has a natural almost para-complex structure $F$ with the eigenbundles $L_+=\mathcal{H},L_-=\mathcal{V}$, where $\mathcal{V}$ is tangent to the fibers of $TN$ and $\mathcal{H}$ is the horizontal distribution of the Levi-Civita connection $\nabla^\mu$ of $\mu$ (e.g., \cite{{AMT},{EST}}). A tangent vector $X\in TN$ has a horizontal and a vertical lift defined by
$$ X^{hor}=X^i\left(\frac{\partial}{\partial x^i}-\Gamma^k_{ij}\dot{x}^j
\frac{\partial}{\partial \dot{x}^k}\right),\; X^{vert}=X^i\frac{\partial}{\partial \dot{x}^i},$$ where $x^i$ are local coordinates on $N$, $\dot{x}^i$ are the corresponding natural coordinates along the fibers of $TN$, $X=X^i(\partial/\partial x^i)$ and $\Gamma^k_{ij}$ are the Christoffel symbols. Then, $$FX^{hor}=X^{hor},\, FX^{vert}=-X^{vert},$$ there exists an $F$-compatible metric $\gamma$ on $M$ given by
$$\gamma(X^{hor},Y^{hor})=0,\; \gamma(X^{vert},Y^{vert})=0,\;
\gamma(X^{hor},Y^{vert})=\mu(X,Y),$$ and $(TN,\gamma ,F)$ is an almost para-Hermitian manifold.

A second almost para-complex structure $H$ is given by $$HX^{hor}=X^{vert},\;HX^{vert}=X^{hor}$$ and the second condition (\ref{ggammaH}) holds. Therefore, $(TN,\gamma,F,H)$ is a (pseudo-)Riemannian, almost para-Hermitian manifold. The compatible metric $g$ defined by $H$ is
$$g(X^{hor},Y^{hor})=\mu(X,Y),\;g(X^{vert},Y^{vert})=\mu(X,Y),\;
g(X^{hor},Y^{vert})=0$$ and it is the so-called Sasaki metric on $TM$. If the metric $\mu$ is flat, $H$ is integrable and we are in the case of a (pseudo-)Riemannian, para-K\"ahler manifold endowed with the Sasaki field $g$.}
\end{example}
\begin{example}\label{prspace} {\rm\cite{GA} Take $M=W\oplus W^*$, where $W$ is a real $n$-dimensional vector space and $W^*$ is its dual space. $M$ has a structure of para-Hermitian vector space with the para-complex structure $F|_W=Id,F|_{W^*}=-Id$ and the metric
$$\gamma((X,\xi),(Y,\eta))= \eta(X)+\xi(Y),\;\; X,Y\in W,\xi,\eta\in W^*.$$ The corresponding fundamental form is
$$\omega((X,\xi),(Y,\eta))=\xi(Y)- \eta(X).$$
Let $(x^i, y_i)_{i=1}^{n}$ be coordinates with respect to dual bases $(e_i)$, $(\epsilon^i)$ in $W,W^*$. Then, the tangent spaces of $M$ may be identified with $W\oplus W^*$ by sending $x^ie_i,y_i\epsilon^i$ to
$x^i(\partial/\partial x^i),y_i(\partial/\partial y_i)$ and we get the para-K\"ahler structure
\begin{equation}\label{tensinproj}
F\frac{\partial}{\partial x^i}=\frac{\partial}{\partial x^i}, F\frac{\partial}{\partial y_i}=-\frac{\partial}{\partial y_i},\; \gamma=dx^i\otimes dy_i,\;
\omega=-dx^i\wedge dy_i.
\end{equation}
We can construct a compatible metric $g$ on $M$ starting with a corresponding pair $(g_-,\beta_-)$ (see Section 2) where $g_-$ is induced on $W^*$ by a metric $s$ on $W$. Then, formulas (\ref{gcumatricea}) give
$$\begin{array}{l}
g((0,\xi),(0,\eta))=s^{-1}(\xi,\eta),\;
g((0,\xi),(X,0))=\beta_-(\xi,\flat_sX),\vspace*{2mm}\\
g((X,0),(Y,0))=s(X,Y)+\beta_-(\flat_sX,\flat_\beta Y)\end{array}$$
($\beta$ is defined by $\beta_-$ by (\ref{defk})).}\end{example}
\begin{example}\label{prspace2} {\rm\cite{GA}
The	{\it para-complex projective model} is the manifold $P^{2(n-1)}$ which is the quotient of	 $M\setminus M_0$, where $M$ is the manifold of Example \ref{prspace}, $M_0=\{(X,\xi)\,/\,\xi(X)=0\},$ by the group $Dl(2,\mathds{R})$ of real, non degenerate, diagonal $2\times2$ matrices acting by $(X,\xi)\mapsto(aX,b\xi)$, $ab\neq0$. Formulas (\ref{tensinproj}) show that $F$ and the metric
$ \tilde{\gamma}=(1/x^iy_i)\gamma$
are invariant under the action of $Dl(2,\mathds{R})$, therefore, they project to a para-Hermitian structure on $P$ whose fundamental form is the projection of $\tilde{\omega}=-dx^i\wedge dy_i/x^iy_i$.
The latter projection is closed and the obtained structure of $P$ is para-K\"ahlerian. Indeed, one also has
$$P=M_+/Dl_+(2,\mathds{R})=S/\mathds{R}_+,$$
where $M_+=\{(X,\xi)\,/\,\xi(X)>0\}$, $Dl_+(2,\mathds{R})$ is the subgroup of $Dl(2,\mathds{R})$ with $a>0,b>0$, $S=\{(X,\xi)\,/\,\xi(X)=1\}$ and $\mathds{R}_+$ is seen as the subgroup of $Dl(2,\mathds{R})$ such that $a>0,b>0,ab=1$. Accordingly, the metric induced on $P$ by $\tilde{\gamma}$ is also the pullback of $\tilde{\gamma}$ to the quotient of $S$ by $\mathds{R}_+$, and, under the restriction $\xi(X)=x^iy_i=1$, we get $d\tilde{\omega}=0$.

Furthermore, take the form $\beta_-=0$ and the metric $g_-=s^{-1}$ for $$s=(1/(s_{ij}x^ix^j))(s_{ij}dx^i\otimes dx^j)$$ where $s_{ij}=const.$ are the components of a positive definite metric on $W$.
We may define $s$ in this way since $x^iy_i\neq0$ implies that not all $x^i$ vanish and $s$ is invariant by $(x^i)\mapsto(ax^i)$, $a\neq0$. The corresponding $\tilde{\gamma}$-compatible metric of $M_+$ is given by
$$g=\frac{s_{ij}dx^i\otimes dx^j}{s_{ij}x^ix^j}
+(s_{hl}x^hx^l)s^{ij}dy_i\otimes dy_j.$$ The metric $g$ is not invariant by the action of $Dl_+(2,\mathds{R})$. But its pullback to $S$ is invariant by the action of $\mathds{R}_+$ on $S$, hence, it projects to a metric of $P$, which together with the projections of $F$ and of the pullback of $\tilde{\gamma}$ to $S$ make $P$ into a Riemannian para-K\"ahler manifold.}\end{example}
\noindent
{\footnotesize{\bf Addendum}.
The name ``projective model" comes from the fact that its construction is related with the construction of the {\it para-complex projective space} (e.g., \cite{Lib}) which is worthy to be recalled.
Denote by $\hspace{1pt}_\mathbf{p}\hspace{-1pt}\mathbf{C}$ the algebra of para-complex numbers defined by
\begin{equation}\label{paracnumb}
\hspace{1pt}_\mathbf{p}\hspace{-1pt}\mathbf{C}=
\{\mathbf{z}=
p+q\mathbf{h}\,/\,p,q\in\mathds{R},\,
\mathbf{h}^2=1\}
\approx\mathfrak{a}l(2,\mathds{R}):=
\left\{\left(\begin{array}{cc}p&q\vspace{2mm}\\ q&p\end{array}\right)\right\}.
\end{equation}
(Other names encountered in the literature are split-complex numbers, Study numbers, hyperbolic numbers, etc., and other notations are $\mathds{A},\mathds{H}$, etc.)
$\hspace{1pt}_\mathbf{p}\hspace{-1pt}\mathbf{C}$ is endowed with the conjugation involution $\bar{\mathbf{z}}
=p-q\mathbf{h}$ and $\mathbf{z}$ is invertible iff $\mathbf{z}\bar{\mathbf{z}}= p^2-q^2\neq0$. Denote by $\hspace{1pt}_\mathbf{p}\hspace{-1pt}\mathbf{C}_0$ the subset of non-invertible para-complex numbers and by $\hspace{1pt}_\mathbf{p}\hspace{-1pt}\mathbf{C}_{inv}$ the subset of invertible numbers; the latter is isomorphic to the multiplicative group $\mathfrak{A}l(2,\mathds{R})$ of non degenerate matrices of the form (\ref{paracnumb}). If we define new components
$$\alpha=p+q,\beta=p-q,$$ we get an algebra isomorphism
\begin{equation}\label{algB}
\hspace{1pt}_\mathbf{p}\hspace{-1pt}\mathbf{C}\approx\mathds{B}:=
\mathds{R}^2 \end{equation} seen as the direct sum of two copies of $\mathds{R}$ with its usual algebra structure (in particular $(\alpha,\beta)\cdot(\alpha',\beta')=(\alpha\alpha',\beta\beta')$). Conjugation now means $ \overline{(\alpha,\beta)}=(\beta,\alpha)$ and the invertibility condition becomes $\alpha.\beta\neq0$; the set of non-invertible elements will be denoted by $\mathds{B}_0$ and that of invertible elements by $\mathds{B}_{inv}$.

The $n$-dimensional, {\it para-complex projective space} is defined by imitating the definition of the real and complex projective spaces as follows. Put
$$\hspace{1pt}_\mathbf{p}\hspace{-1pt}\mathbf{C}\mathbf{P}^n=
\{[(z^1,...,z^{n+1})]
\,/\,z^i\in\hspace{1pt}_\mathbf{p}\hspace{-1pt}\mathbf{C},\,
\sum_{i=1}^{n+1}|z^i\bar{z}^i|\neq0\},$$
where the bracket denotes the equivalence class by the equivalence relation
$$(z^1,...,z^{n+1})\approx\lambda(z^1,...,z^{n+1}),\;\lambda\in \hspace{1pt}_\mathbf{p}\hspace{-1pt}\mathbf{C}_{inv}.$$ The condition
$\sum_{i=1}^{n+1}|z^i\bar{z}^i|\neq0$ means that at least one of the {\it homogeneous coordinates} $z^i$ is invertible, i.e., we have
$$\hspace{1pt}_\mathbf{p}\hspace{-1pt}\mathbf{C}\mathbf{P}^n=
(\hspace{1pt}_\mathbf{p}\hspace{-1pt}\mathbf{C}^{n+1}
\setminus \hspace{1pt}_\mathbf{p}\hspace{-1pt}\mathbf{C}_0^{n+1})/\{v\approx\lambda v,\,
\forall v\in (\hspace{1pt}_\mathbf{p}\hspace{-1pt}\mathbf{C}^{n+1}
\setminus \hspace{1pt}_\mathbf{p}\hspace{-1pt}\mathbf{C}_0^{n+1}),
\,\forall \lambda\in \hspace{1pt}_\mathbf{p}\hspace{-1pt}\mathbf{C}_{inv}\}$$
and, with the isomorphism (\ref{algB}), we have
$$\hspace{1pt}_\mathbf{p}\hspace{+1pt}\mathbf{C}\mathbf{P}^n=
(\mathds{B}^{n+1}\setminus \mathds{B}_0^{n+1})/\{\zeta\approx\mu\zeta,\,
\forall\zeta\in(\mathds{B}^{n+1}\setminus \mathds{B}_0^{n+1}),\,\mu\in\mathds{B}_{inv}\}.$$ By taking the components $\zeta=(a,b)$, we get an identification $\hspace{1pt}_\mathbf{p}\hspace{+1pt}\mathbf{C}
\mathbf{P}^n\approx\mathds{R}P^n
\times\mathds{R}P^n$.

Now, we restrict the previous construction and define
$$\mathds{B}\mathbf{P}^n=\{[(z^1,...,z^{n+1})]
\,/\,z^i\in\hspace{1pt}_\mathbf{p}\hspace{-1pt}\mathbf{C},\,
\sum_{i=1}^{n+1}z^i\bar{z}^i\neq0\}$$ $$\approx\{[(\alpha^1,\beta^1),...,
(\alpha^{n+1},\beta^{n+1})]\,/\sum_{i=1}^{n+1}\alpha^i\beta^i\neq0\},$$
where, in the second expression, the bracket denotes the equivalence class by the equivalence relation defined by multiplication with $\mu=(\kappa,\sigma)$ such that $\kappa\sigma\neq0$.
The set that defines $\mathds{B}\mathbf{P}^n$ is a	subset of the set that defines $\hspace{1pt}_\mathbf{p}\hspace{-1pt}\mathbf{C}\mathbf{P}^n$. Obviously, the space $\mathds{B}\mathbf{P}^n$ is the projective model defined by the linear space $W=\mathds{R}^{n+1}$.

Notice also that the same space $\mathds{B}\mathbf{P}^n$ is obtained if we ask $\sum_{i=1}^{n+1}z^i\bar{z}^i>0$ ($\sum_{i=1}^{n+1}\alpha^i\beta^i>0$) and ask $\mu=(\alpha,\beta)$ to satisfy the condition $\alpha>0,\beta>0$, which is equivalent to asking the factor $\lambda=l+s\mathbf{h}$ to satisfy the conditions $l>0,l^2-s^2>0$. Indeed, if $\sum_{i=1}^{n+1}\alpha^i\beta^i<0$, we may multiply by the factor $(1,-1)$ and get an equivalent element with $\sum_{i=1}^{n+1}\alpha^i\beta^i>0$.

The spaces $\hspace{1pt}_\mathbf{p}\hspace{-1pt}\mathbf{C}\mathbf{P}^n$ and $\mathds{B}\mathbf{P}^n$are covered by	$n+1$ domains of para-complex, non-homogeneous coordinates, e.g., $\mathbf{u}^i_{(n+1)}=\mathbf{z}^i/\mathbf{z}^{n+1}$, $i=1,...,n$, on the domain $\mathbf{z}^{n+1}\in \hspace{1pt}_\mathbf{p}\hspace{-1pt}\mathbf{C}_{inv}$, etc. This shows that $\hspace{1pt}_\mathbf{p}\hspace{-1pt}\mathbf{C}\mathbf{P}^n$ and $\mathds{B}\mathbf{P}^n$ are real, differentiable manifolds of dimension $2n$.}
{\small Department of
Mathematics, University of Haifa, Israel,
vaisman@math.haifa.ac.il}
\end{document}